\documentclass[reqno,a4paper,12pt]{amsart}
\usepackage{amsfonts}
\usepackage{amsmath}
\usepackage{amssymb}
\usepackage{amsthm,mathtools}
\usepackage{mathabx}
\usepackage[margin=1.3in]{geometry}
\usepackage{hyperref}
\usepackage{color}
\usepackage{xcolor}
\usepackage{mathrsfs}
\usepackage{verbatim, graphicx, ifthen, enumitem}
\usepackage[T1]{fontenc}
\usepackage [applemac] {inputenc}
\usepackage{accents}
\usepackage{amsrefs}

\allowdisplaybreaks
\newtheorem{theorem}{Theorem}[section]
\newtheorem{proposition}[theorem]{Proposition}
\newtheorem{corollary}[theorem]{Corollary}
\newtheorem{lemma}[theorem]{Lemma}

\theoremstyle{definition}

\newtheorem{remark}[theorem]{Remark}
\newtheorem{definition}[theorem]{Definition}
\theoremstyle{definition}

\def\N{{\mathbb{N}}}
\def\R{{\mathbb{R}}}
\def\Z{{\mathbb{Z}}}
\def\T{{\mathbb{T}}}
\def\C{{\mathbb{C}}}
\def\MM{{\mathcal{M}}}
\renewcommand{\v}[1]{\ensuremath{\mathbf{#1}}}

\title{Uncertainty Principles for Fourier Multipliers}

\author[M.~Northington]{Michael Northington V}
\address{School of Mathematics, Georgia Institute of Technology, Atlanta, GA 30332}
\email{mcnv3@gatech.edu}
\thanks{The author is supported by NSF grant DMS-1344199.  Work on this paper was also supported by  NSF grants DMS-1600726 and DMS-1521749.}

\begin{document}

\begin{abstract}
The admittable Sobolev regularity is quantified for a function, $w$, which has a zero in the $d$--dimensional torus and whose reciprocal $u=1/w$ is a $(p,q)$--multiplier.
Several aspects of this problem are addressed, including zero--sets of positive Hausdorff dimension, matrix valued Fourier multipliers, and non--symmetric versions of
Sobolev regularity.  Additionally, we make a connection between Fourier multipliers and approximation properties of Gabor systems and shift--invariant systems.  We exploit this connection and the results on Fourier multipliers to refine and extend versions of the Balian--Low uncertainty principle in these settings.
\end{abstract}

\maketitle

\section{Introduction}\label{sec:Introduction}
Let $\mathcal{F}$ denote the Fourier transform on $L^1(\T^d)$. Given $2\leq q\leq \infty$, a function $u\in L^2(\T^d)$ is called a $(2,q)$--\textit{Fourier multiplier}, or $(2,q)$--\textit{multiplier} for short, if the operator $T_u$ defined by
$$
T_u a= \mathcal{F}(u\mathcal{F}^{-1}{a})
$$
is a bounded operator from  $\ell^2(\Z^d)$ to $\ell^q(\Z^d)$. The family of all $(2,q)$--multipliers is denoted by $\mathcal{M}_2^q$ and is a Banach space when
endowed with the operator norm $\|u\|_{\MM_2^q}= \|T_u\|_{\ell^2(\Z^d)\rightarrow \ell^q(\Z^d)}$.
Fourier multipliers are a classical subject in analysis (see e.g. \cite{Ma, Mi}). The study of Fourier multipliers from a $p$--normed space to a different $q$--normed space goes back to A.~Devinatz and I.~I.~Hirschman,~Jr \cite{DH,Hi} and to L.~H\"ormander \cite{H}.

A function $u\in L^{\infty}(\T^d)$ is clearly a $(2,q)$--multiplier for every value of $q \geq 2$. In fact, it is readily checked that $\mathcal{M}_2^2=L^{\infty}(\T^d)$. However, if $u$ is not bounded then the situation becomes more delicate, and one may suspect that if $u$ `grows rapidly' near its `singularities' then $u$ will not be a $(2,q)$--multiplier, at least for certain values of $q$.

All results in this paper are joint work with Shahaf Nitzen and Alex Powell. The goals of this paper are twofold. First, as described above, we study growth restrictions on a $(2,q)$--multiplier $u$. More precisely, we assume that
$u=1/w$ for some function $w$ with a zero in the $d$--dimensional torus, and we quantify how smooth $w$ can be in the sense of Sobolev regularity.
Our results extend in several directions including: zero sets of positive Hausdorff dimension, general $(p,q)$--multipliers, non--symmetric versions of Sobolev regularity,
and matrix valued Fourier multipliers.

Our second goal is to developed a machinery by which one can relate such results on Fourier multipliers to problems on the uncertainty principle in time--frequency analysis.
We investigate trade--offs between the time--frequency localization of the window function $g$ and the approximation properties of the associated Gabor system, that is, the collection of integer translates and modulates of $g$.
Our motivation comes from the classical Balian--Low theorem \cite{Bal81, Bat88, Dau, Low} which quantifies the time--frequency localization of generators of Gabor systems
that are Riesz bases for $L^2(\R)$.
We use our results on Fourier multipliers to provide sharp refinements of Balian--Low theorems obtained in \cite{NO1}.
Our approach also yields new Balian--Low type theorems for systems of translates in shift--invariant spaces.

\subsection{Restrictions on \texorpdfstring{$(2,q)$--multipliers}{(2,q)--multipliers}} \label{sec:restric-summary}
Given $s\geq 0$, recall that $w\in L^2(\T^d)$ belongs to the Sobolev space $W^{s,2}(\T^d)$ if and only if $\widehat{w} \in \ell^2(\Z^d)$ satisfies
$\sum_{k\in\Z^d}|k|^{2s}|\widehat{\omega}(k)|^2<\infty.$

\subsubsection{The case of a single zero}\label{1.1.1}

Our first result is the following.

\begin{theorem}[Nitzan, Northington, Powell]\label{thm:SobolevResultSingleZero-2-2}
 Let $\frac{d}{2}<s\le d$ and $w \in W^{s,2}(\T^d)$.  Suppose that $w$ has a zero.
\begin{itemize}
	\item[i.] If $s<\frac{d}{2}+1$, then $u=\frac{1}{w} \notin \mathcal{M}_2^q$ for any $q$ satisfying $2 \le q \le \frac{d}{d-s}$.  Conversely, for any $q>\frac{d}{d-s}$, there exists $w \in W^{s,2}(\T^d)$ such that $w$ has a zero and $u=\frac{1}{w}\in \mathcal{M}_2^q$.
	\item[ii.] If $s=\frac{d}{2}+1$, then $u=\frac{1}{w} \notin \mathcal{M}_2^q$ for any $q$ satisfying $2 \le q<\frac{2d}{d-2}$.
	\end{itemize}
\end{theorem}

The assumption $s>\frac{d}{2}$ in Theorem \ref{thm:SobolevResultSingleZero-2-2} implies that $W^{s,2}(\T^d)$ embeds into $C(\T^d)$ (see e.g. Section \ref{sec:Sobolevdef}),
and so the assumption that  `$w$ has a zero' should be interpreted to mean that the continuous representative of $w$ has a zero. Further, the relation $ \mathcal{M}_2^2=L^{\infty}(\T^d)$ implies that Theorem \ref{thm:SobolevResultSingleZero-2-2} also holds in the case $s=\frac{d}{2}$, where $w$ need not have a continuous representative and where the condition `$w$ has a zero' requires suitable interpretation, see Section \ref{subsec:LargerZeroSets}.
\begin{remark}\label{rem:Wsr}
	A version of Theorem \ref{thm:SobolevResultSingleZero-2-2} holds when the more general Sobolev spaces  $W^{s,r}(\T^d)$, with $1<r <\infty$, are considered. See Theorem \ref{thm:SobolevResultSingleZero-r-2}. (See also Section \ref{sec:Sobolevdef} for the definition of the spaces  $W^{s,r}(\T^d)$).
\end{remark}
\begin{remark}\label{rem:Wnonsymm}
A version of Theorem \ref{thm:SobolevResultSingleZero-2-2} holds when the assumption $w \in W^{s,2}(\T^d)$ is replaced by non--symmetric smoothness conditions where the function has different Sobolev regularities in each variable, see Theorem \ref{thm:nonsymMult}. (See Section \ref{sec:Sobolevdef} for the definition of these spaces)
\end{remark}
\begin{remark}\label{rem:pqMultipliers}
For  $1\leq p\leq q\leq 2$, or $2\leq p\leq q\leq\infty$, Theorem \ref{thm:SobolevResultSingleZero-2-2} can be extended to the more general $(p,q)$--multipliers, see the discussion in Section \ref{subsec:PQMult}.
\end{remark}

\subsubsection{Zero sets of positive Hausdorff dimension}\label{subsec:LargerZeroSets}

Recall that for $S\subset\R^d$ and $\sigma>0$ the $\sigma$--dimensional Hausdorff measure of $S$ is defined by
\begin{equation}\label{eqn:HausDim}
\mathcal{H}^\sigma(S)=\lim_{\delta\rightarrow 0}\inf \left\{ \sum_i \text{diam}(U_i)^\sigma: S \subset \bigcup U_i \textrm{ where the sets $U_i$ satisfy } \text{diam}(U_i)<\delta\right\}.
\end{equation}

Our next result extends Theorem \ref{thm:SobolevResultSingleZero-2-2} to zero sets of positive Hausdorff dimension. To allow the consideration of non--continuous functions, we define the zero set of a function $w\in L^1(\T^d)$ to be
\begin{align}	\label{eqn:ZeroSetDef}
	\Sigma(w)&=\left\{x \in \T^d: \limsup_{\tau\rightarrow 0} \frac{1}{|I_\tau(x)|} \int_{I_\tau(x)} |w(y)| dy = 0\right\},
\end{align}
where $I_\tau(x)$ is the cube of width $2\tau$ centered at $x$.  We note that part (ii) of Theorem \ref{thm:ScalarValuedLargeZeroResult} follows from the main results of \cite{JL, SC}.

\begin{theorem}[Nitzan, Northington, Powell]\label{thm:ScalarValuedLargeZeroResult}
 Let $(d-\sigma)/2 <s \le 1$ and $w \in W^{s,2}(\T^d)$. Suppose that $\mathcal{H}^\sigma(\Sigma(w))>0$.
\begin{itemize}
	\item[(i)] If $s < d-\sigma$, then $u=\frac{1}{w} \notin \mathcal{M}_2^q$ for any $q$ satisfying $2 \le q \le \frac{d}{d-s-\frac{\sigma}{2}}$.
	\item[(ii)] If {$s=d-\sigma$}, then $u=\frac{1}{w} \notin L^2(\T^d)$, and thus $u \notin \mathcal{M}_2^q$ for any $q\geq 2$.
	\end{itemize}
\end{theorem}
\begin{remark}
 A version of Theorem \ref{thm:ScalarValuedLargeZeroResult} holds for $s\ge 1$, under some mild additional requirements.  See Remark \ref{rem:PoincareIneq-s} and Theorem \ref{thm:ScalarValuedLargeZeroGeneralResult}.
\end{remark}
\begin{remark}
Remarks \ref{rem:Wsr} and \ref{rem:pqMultipliers} hold also for Theorem \ref{thm:ScalarValuedLargeZeroResult}. See Theorem \ref{thm:ScalarValuedLargeZeroGeneralResult} and the discussion in Section \ref{subsec:PQMult}.
\end{remark}

Theorems \ref{thm:ScalarValuedLargeZeroResult} and  \ref{thm:ScalarValuedLargeZeroGeneralResult} should be compared to the main results in \cite{JL, SC}.
There,  H.~Jiang, F.~Lin, and A.~Schikorra
relate the Hausdorff dimension of the zero set of a function to the integrability properties of its inverse. More precisely, for  $0\le s\le 1$, they prove that if $w \in W^{s,2}(\T^d)$  and $\mathcal{H}^{\sigma}(\Sigma(w))>0$ then $1/w\notin L^{\frac{2q}{q-2}}(\T^d)$ whenever $q\le \frac{d-\sigma}{d-\sigma-s}$. The conditions $1/w \notin \mathcal{M}_2^q$ and $1/w\notin L^{\frac{2q}{q-2}}(\T^d)$ are related in one direction via the Hausdorff--Young inequality. Indeed, it is readily proved that $L^{\frac{2q}{q-2}}(\T^d)\subset \mathcal{M}_{2}^q(\T^d)$. The reverse inclusion does not hold (see e.g. Theorem 1.4.24 in \cite{Gr}).

In contrast to Theorem \ref{thm:SobolevResultSingleZero-2-2}, Theorem \ref{thm:ScalarValuedLargeZeroResult} does not appear to be sharp, and it is reasonable to ask if the range of the parameter $q$ in Theorem \ref{thm:ScalarValuedLargeZeroResult} can be extended to the
range $q\le \frac{d-\sigma}{d-\sigma-s}$ that appears in  \cite{JL, SC}.  For further discussion on this, see Section \ref{open problems}.

\subsubsection{Matrix valued Fourier multipliers}

Given a Banach space $B$ and $K\in\N$, let $[B]^K$ denote the space of all $K$--tuples with elements from $B$, and let $[B]^{K\times K}$ denote the space of all $K\times K$ matrices with elements from $B$.  That is, $[B]^K=\{G=(g_k)_{k=1}^K:g_k\in B\}$ and similarly for $[B]^{K\times K}$. In the cases where $B$ is an $L^p$ or $\ell^p$ space, we endow the spaces $[B]^K$ with the natural norms
\[ \|G\|_{[L^{q}]^K}=\left( \sum_{k=1}^K \|g_{k} \|_{L^q}^q  \right)^{1/q} \quad\textrm{and}\quad \|G\|_{[\ell^{q}]^K}=\left( \sum_{k=1}^K \|g_{k} \|_{\ell^q}^q \right)^{1/q}.\]
The Fourier transform $\mathcal{F}_K$ on $[L^1(\T^d)]^K$ is defined coordinate--wise by $\mathcal{F}_KG=(\mathcal{F} g_k)_{k=1}^K$. Note that $\|G\|_{[L^2]^K}=\|\mathcal{F}_K G\|_{[\ell^{2}]^K}$.

Given $2\leq q\leq\infty$, we say that $U\in [L^2(\T^d)]^{K\times K}$ is a \textit{matrix valued $(2,q)$--multiplier} if the operator $T_U: [\ell^2(\Z^d)]^K\rightarrow [\ell^q(\Z^d)]^K$ defined by
$$
T_U A= \mathcal{F}_K(U\mathcal{F}^{-1}_K{A})
$$
is bounded. We denote by $\mathfrak{M}_2^q(K)$ the family of all matrix valued $(2,q)$--multipliers.

For a Hermitian matrix valued function $U\in [L^2(\T^d)]^{K\times K}$, there exist matrix valued functions $V$ and $\Lambda$ such that $U=V^* \Lambda V$, where the entries of $V$ and $\Lambda$ are measurable functions, $V$ is unitary, and $\Lambda$ is a diagonal matrix with diagonal entries satisfying $\lambda_1(x)\ge \lambda_2(x) \ge \cdots \ge \lambda_K(x)$ for almost every $x\in \T^d$, (see e.g. Lemma 2.3.5 in \cite{RS}).  Throughout the paper, when we write $U=V^* \Lambda V$, we mean that $V$ and $\Lambda$ satisfy the conditions described above.

Our next result relates scalar valued multipliers and matrix valued multipliers.
\begin{theorem}[Nitzan, Northington, Powell]\label{thm:KMultEig}
Let $U\in [L^2(\T^d)]^{K\times K}$ be Hermitian with eigenvalues $\lambda_1(x)\ge \lambda_2(x)\ge \cdots \ge \lambda_K(x)$.  Then, $U \in \mathfrak{M}_2^q(K)$ if and only if $\lambda_k \in \mathcal{M}_2^q$  for each $1 \leq k \leq K$.
\end{theorem}

Theorem \ref{thm:KMultEig} allows the extension of Theorems \ref{thm:SobolevResultSingleZero-2-2} and \ref{thm:ScalarValuedLargeZeroResult} to matrix valued multipliers.

\begin{corollary}[Nitzan, Northington, Powell]\label{thm:SobolevResultSingleZero-2-2-matrix}
Let $\frac{d}{2}<s\le d$ and let $W \in [L^1(\T^d)]^{K\times K}$ be Hermitian with eigenvalues $\lambda_1(x)\ge \lambda_2(x) \ge \cdots \ge \lambda_K(x)$. If $\lambda_k\in W^{s,2}(\T^d)$ for every $k=1,...,K$, and $\det(W)$ has a zero, then conclusions (i) and (ii) of Theorem \ref{thm:SobolevResultSingleZero-2-2} hold with $U=W^{-1}$ replacing $u=1/w$ and with $\mathfrak{M}_2^q(K)$ replacing $\mathcal{M}_2^q$. In particular, the condition `$\lambda_k \in W^{s,2}(\T^d)$ for every $k$' may be replaced by the condition `$W$ is a nonnegative matrix valued function satisfying  $W \in [W^{s,2}(\T^d)]^{K\times K}$'.
\end{corollary}

\begin{corollary}[Nitzan, Northington, Powell]\label{thm:ScalarValuedLargeZeroResult-matrix}
Let $ (d-\sigma)/2 <s\le 1$ and let $W \in [L^1(\T^d)]^{K\times K}$ be Hermitian with eigenvalues given by $\lambda_1(x)\ge \lambda_2(x) \ge \cdots \ge \lambda_K(x)$. If $\lambda_k\in W^{s,2}(\T^d)$  for every $k=1,...,K$, and $\mathcal{H}^\sigma(\Sigma(\det(W))>0$, then conclusions (i) and (ii) of Theorem \ref{thm:ScalarValuedLargeZeroResult} hold with $U=W^{-1}$ replacing $u=1/w$ and with $\mathfrak{M}_2^q(K)$ replacing $\mathcal{M}_2^q$. In particular, the condition `$\lambda_k \in W^{s,2}(\T^d)$ for every $k$' may be replaced by the condition `$W$ is a nonnegative matrix valued function satisfying  $W \in [W^{s,2}(\T^d)]^{K\times K}$'.
 \end{corollary}

\subsection{Applications to time--frequency analysis}\label{subsec:TFApps}

We now turn to discuss applications of the above stated results to time--frequency analysis.

Recall that a system $\{f_n\}$ in a separable Hilbert space $H$ is a \textit{Riesz basis}
if it is the image of an orthonormal basis under a bounded and invertible operator, that is, if it is  complete in $H$ and there exist positive constants
$A$ and $B$ such that for all finite sequences $\{a_n\}$,
\begin{equation}\label{riesz-cond}
A\sum |a_n|^2\leq \left\|\sum a_nf_n\right\|^2_H\leq B\sum |a_n|^2.
\end{equation}
Note that $\{f_n\}$ is a Riesz basis if and only if every $f\in H$ can be decomposed in a unique way into a series $f=\sum a_nf_n$
with a norm equivalence between the $\ell^2$ norm of the coefficients $\{a_n\}$ and the norm of $f$.
Any system for which the right hand inequality in (\ref{riesz-cond}) holds is called a \textit{Bessel system}.

Further, a system $\{h_n\} \in H$ is called {\em minimal} if each of its elements lies outside the closed linear span of the remaining elements.
A system which is both complete and minimal is called \emph{exact}.  Every Riesz basis is exact, but the converse is not true.
For further background on Riesz bases, Bessel systems and exact systems see \cite{HPrimer}.

\subsubsection{Generators of Gabor systems}

Given $g\in L^2(\R)$, the {\em Gabor system} generated by $g$ over the lattice $\Z^2$ is defined by
\[
G(g)= \{e^{2\pi i m x}g(x-n)\}_{(m,n) \in \Z^2}.
\]
The Balian--Low theorem, \cite{Bal81,Low,Dau,Bat88},  states that if $G(g)$ is a Riesz basis for $L^2(\R)$, then $g$ must have much worse time--frequancy localization then what the uncertainty principle permits. More precisely, for $t \geq 2$
\begin{align}\label{eqn:CqGaborBLTIntegrals}
\int_{\R} |x|^{t} |g(x)|^2 dx = \infty  \ \ \ \text{   or   } \ \ \ \int_{\R} |\xi|^{t} |\widehat{g}(\xi)|^2 d\xi = \infty.
\end{align}
This result is sharp in the sense that the conclusion \eqref{eqn:CqGaborBLTIntegrals} fails if $t<2$, see \cite{BCGP}.

The Balian--Low theorem has inspired a large body of work during the last 25 years and has been extended in many directions.
In particular, I.~Daubechies and A.~J.~E.~M.~Janssen prove in \cite{DJ} that if $G(g)$ is merely an exact system then  (\ref{eqn:CqGaborBLTIntegrals}) holds whenever $t\geq 4$, and this is sharp. Further, in \cite{NO1} S.~Nitzan and J.--F.~Olsen give a collection of Balian--Low theorems which interpolate the results for Riesz bases and for exact systems,
by considering the intermediate class of the so called \textit{exact ($C_q$)--systems}.

Given $q\geq 2$, a system $\{f_n\}$ in a Hilbert space $H$ is called a {\em ($C_q$)--system}
if there exists $C>0$ such that every $f\in H$ can be approximated arbitrarily well by a finite linear combination $\sum a_nf_n$ with $\|a_n\|_{\ell^q}\leq C\|f\|_H$.
The study of ($C_q$)--systems originated in \cite{NHO}, where S.~Nitzan and A.~Olevskii studied possible density restrictions on exponential systems of this type. The name `($C_q$)--system'  emphasizes that such systems are `\textit{complete with $\ell^q$ control on the coefficients}'.
We note that $\{f_n\}$ is an exact ($C_q$)--system if and only if it is complete and there exists $D>0$ such that
\begin{equation}\label{Cq-cond}
D \left(\sum|a_n|^q\right)^{\frac{1}{q}}\leq \left\|\sum a_nf_n \right\|_H,
\end{equation}
for any finite sequence $\{a_n\}$, \cite{NIT} (see also Theorem 3 in \cite{NO1}). This condition should be compared with (\ref{riesz-cond}).

It is readily checked that if $\widetilde{q}>q$ then every $(C_{\widetilde{q}})$--system is also a $(C_{q})$--system, and that an exact system is a Bessel ($C_2$)--system if and only if it is a Riesz basis. Finally, we note that a Gabor system $G(g)$ is an exact system if and only if it is an exact ($C_{\infty}$)--system. With this we conclude that for Gabor systems, the family of exact ($C_q$)--systems provides a family of systems that range between Riesz bases and exact systems as $q$ ranges between two and infinity.

In \cite{NO1}, S.~Nitzan and J.--F.~Olsen prove that if $2\leq q\leq \infty$ and $G(g)$ is an exact ($C_q$)--system, then (\ref{eqn:CqGaborBLTIntegrals}) holds for $t > 4/q'$ where $q'=q/(q-1)$.  The result obtained in \cite{NO1} is almost sharp in the sense that for every  $t < 4/q'$ there exists $g\in L^2(\R)$ for which $G(g)$ is an exact ($C_q$)--system and both integrals in  (\ref{eqn:CqGaborBLTIntegrals}) converge.  However, the critical--exponent case of $t = 4/q'$ remained unsettled.  We resolve this critical case with the following theorem.

\begin{theorem}[Nitzan, Northington, Powell]\label{thm:CqGabor}
Fix $q> 2$.  Let $g \in L^2(\R)$.  If $G(g)$ is an exact $(C_q)$--system for $L^2(\R)$, then
(\ref{eqn:CqGaborBLTIntegrals}) holds whenever $t \ge 4/q'$, where $q'=q/(q-1)$. This result is sharp in the sense that it fails when $t < 4/q'$.
\end{theorem}

\begin{remark} Following \cite{NO1} (see also \cite{Gautam}, \cite{HPExact}), in Theorem \ref{thm:nonsymmetricBLT} we prove a non--symmetric extension of Theorem \ref{thm:CqGabor} that replaces \eqref{eqn:CqGaborBLTIntegrals}
by non--symmetric time--frequency constraints
$$\int_{\R} |x|^{r} |g(x)|^2 dx = \infty \ \ \  \text{   or   } \ \ \ \int_{\R} |\xi|^{t} |\widehat{g}(\xi)|^2 d\xi = \infty.$$
This theorem provides critical--exponent versions of a large portion of the non--symmetric results in \cite{NO1}.
\end{remark}

\subsubsection{Generators of shift--invariant spaces}

Fix $K\in \N$ and $F=\{f_k\}_{k=_1}^K\subset L^2(\R^d)$. The \textit{shift--invariant space generated by $F$}, denoted  $V(F)$, is the closed linear span of the integer translates of the elements of $F$, that is
$$V(F)=\overline{\textrm{span} \thinspace \mathcal{T}(F)},\qquad\mathcal{T}(F)=\{f_k(x-n):n\in\Z^d, k=1,..,K\}.$$
We say that a shift--invariant space $V(F)$ has \textit{extra invariance} if there exists $\gamma \in \R^d \setminus \mathbb{Z}^d$ such that for every $h\in V(F)$ we have also  $h(x-\gamma)\in V(F)$.
We say that this extra invariance is \textit{non--trivial} if $J\gamma  \notin \mathbb{Z}^d$,
 where $J$ is the minimal cardinality of a set $H\subset L^2(\R)$ which satisfies $V(F)=V(H)$. Indeed, in this case the extra invariance is not a trivial consequence of the functions in $F$ being shifts of one another.
For more information on extra invariance in shift--invariant spaces, see \cite{ACHKM, ACP, ACP2}.

Analogues of the Balian--Low theorem for shift--invariant spaces were studied in \cite{ ASW, HNP, TW}.
Specifically, in \cite{HNP} D.~Hardin, A.~M.~Powell, and the author prove that if $V(F)$ admits non--trivial extra invariance, and $\mathcal{T}(F)$ is a Riesz basis for $V(F)$,
then for $t\geq 1$ there exists at least one $f_k\in F$ which satisfies
\begin{equation}\label{shift-inv-blt}
	\int_{\R^d} |x|^{t} |f_k(x)|^2 dx = \infty.
\end{equation}
 This theorem is sharp and the non--triviality of the extra invariance is necessary, \cite{HNP}.
 Our next results extend the Balian--Low theorem for shift--invariant spaces from the setting of Riesz bases to the settings of exact systems and exact $(C_q)$--systems.
Similar to Gabor systems, we note that the system of translates $T(F)$ is exact in $V(F)$ if and only if it is an exact $(C_{\infty})$--system there.
We first formulate our results over $\R$.

\begin{theorem}[Nitzan, Northington, Powell]\label{thm:CqSISHD-1d}
Let $2\le q \le \infty$ and let $F=\{f_k\}_{k=1}^K\subset L^2(\R)$. If $V(F)$ admits non--trivial extra invariance, and $\mathcal{T}(F)$ is an exact $(C_q)$--system for $V(F)$, then (\ref{shift-inv-blt}) holds for some $f_k \in F$ whenever $t\geq 2/q'$, where $q'=q/(q-1)$. The condition $t\geq 2/q'$ is sharp.
\end{theorem}

Next, we extend Theorem \ref{thm:CqSISHD-1d} to higher dimensions.

\begin{theorem}[Nitzan, Northington, Powell]\label{thm:CqSISHD}
Let $2\le q \le \infty$ and let $F=\{f_k\}_{k=1}^K\subset L^2(\R^d)$. If $V(F)$ admits non--trivial extra invariance, and $\mathcal{T}(F)$ is an exact $(C_q)$--system for $V(F)$, then (\ref{shift-inv-blt}) holds for some $f_k \in F$ and $t\geq \min\left(2d/q'-d+1,2\right)$, where $q'=q/(q-1)$. This result is sharp for $q=2$ and $q=\infty$.
\end{theorem}

Theorem \ref{thm:CqSISHD-1d} and the cases $q=2,\infty$ of Theorem \ref{thm:CqSISHD} are both sharp.  In contrast, the cases $2<q<\infty$ of Theorem \ref{thm:CqSISHD} do not appear to be sharp, and it is reasonable to ask if the condition $t\geq \min\left(2d/q'-d+1,2\right)$ can be replaced by $t\geq 2/q'$ for all $2\leq q\leq\infty$, and all dimensions $d$,
see the discussion in Section \ref{open problems}.

\subsection{Outline of the paper}

The paper is organized as follows. In Section \ref{sec:Sobolevdef}, we provide necessary background and preliminary results regarding Sobolev spaces.  In Section \ref{sec:Proofs}, we study the
zero--sets of reciprocals of Fourier multipliers.  We begin by proving
Theorems \ref{thm:SobolevResultSingleZero-r-2} and \ref{thm:ScalarValuedLargeZeroGeneralResult}, and thereby obtaining Theorems \ref{thm:SobolevResultSingleZero-2-2} and \ref{thm:ScalarValuedLargeZeroResult} as special cases.
We next prove Theorem \ref{thm:nonsymMult} which addresses non--symmetric Sobolev spaces.

In Section \ref{sec:MatpqMult}, we extend the results of Section \ref{sec:Proofs} to $(p,q)$--multipliers and to matrix valued multipliers. In particular, we prove Theorem \ref{thm:KMultEig} and Corollaries \ref{thm:SobolevResultSingleZero-2-2-matrix} and \ref{thm:ScalarValuedLargeZeroResult-matrix}.
In Section \ref{sec:TFApp}, we discuss the relation to Gabor systems and use the results obtained in Section \ref{sec:Proofs} to prove Theorem \ref{thm:CqGabor}.   In Section \ref{subsec:SIS} we consider shift--invariant spaces and use the results from Sections \ref{sec:Proofs} and \ref{sec:MatpqMult} to prove Theorems \ref{thm:CqSISHD-1d} and \ref{thm:CqSISHD}.
We conclude by formulating some open problems and possible research directions in Section \ref{open problems}.

\section{Sobolev spaces}\label{sec:Sobolevdef}

Throughout this paper the smoothness of a function is quantified by the function belonging to certain Sobolev spaces. In this section we collect basic definitions and properties of these spaces.

\subsection{Notations} We denote the $d$--dimensional torus by $\T^d=\R^d/\Z^d$ and identify it with the interval $[-1/2,1/2)^d$.
The Fourier transform of $f \in L^1(\R^d)$ and the Fourier coefficients of $g \in L^1(\T^d)$ are normalized as follows,
\[\widehat{f}(\xi)=\int_{\R^d} f(x)e^{-2\pi i \langle x, \xi\rangle} dx ,\ \ \ \ \ \widehat{g}(k)=\int_{\T^d} g(x)e^{-2\pi i \langle k , x\rangle } dx, \]
so that the extension of the Fourier transform to $L^2(\R)$, as well as the restriction of the Fourier coefficients to $L^2(\T^d)$, are unitary operators.

We denote the cube of sidelength $2\tau$ centered at $x\in \R^d$  by $I_\tau(x)= x+[-\tau,\tau]^d$, and the ball of radius $\tau$ centered at the same point by $B_\tau(x)$. For a real number $s$, $\lfloor s\rfloor$ denotes the integer part of $s$, that is, the largest integer less than or equal to $s$, and $\{s\}=s-\lfloor s\rfloor$ denotes the fractional part of $s$. For vectors $x,y \in \C^d$, $\langle x,y\rangle$ denotes the standard inner product, and $|x|$ denotes the corresponding Euclidean norm.  For a vector of nonnegative integers $\alpha=(\alpha_1,...,\alpha_d)$, we denote $|\alpha|_1=\alpha_1+\cdots+ \alpha_d$ and  $D^{\alpha}=D_{1}^{\alpha_1} \cdots D_d^{\alpha_d}$, where $D_j$ is the distributional partial derivative operator with respect to the $j^{th}$ variable.

For a finite set $R$, $\# R$ denotes the cardinality of the set $R$.  For a set $E$ in either $\R^d$ or $\T^d$, $|E|$ denoted the Lebesgue measure of $E$.  The $\sigma$--dimensional Hausdorff measure of $E$ is denoted by $\mathcal{H}^\sigma(E)$ and defined as in \eqref{eqn:HausDim}, and for $w\in L^1(\T^d)$, the generalized zero set of $w$ is denoted by $\Sigma(w)$ and defined as in \eqref{eqn:ZeroSetDef}.

For $0<\alpha\le 1$ and an open subset $E \subset \R^d$, a function $f$ is  $\alpha$--H\"older continuous on $E$ if $f \in L^\infty(E)$ and there exists a positive constant $C$ so that $|f(x)-f(y)|\leq C|x-y|^\alpha$ for every $x,y \in E$.
Similarly, a periodic function $g$ is $\alpha$--H\"older continuous on $\T^d$ if its periodic extension is $\alpha$--H\"older continuous on $\R^d$.

We use the notation $A \lesssim B$ to imply that there exists a constant $c$ such that $A \le c B$.  Similarly, $A \asymp B$ means $A \lesssim B$ and $B \lesssim A$.

\subsection{The Bessel potential spaces and Sobolev--Slobodeckij spaces}
We recall the definitions of the following related classes of Sobolev spaces.  
\begin{definition} 
	Given $s>0$, the Bessel potential space $H^s(\T^d)$ consists of all $f \in L^2(\T^d)$ for which the following semi--norm is finite
\[
	\|f\|_{\dot{H}^s(\T^d)}^2= \sum_{k \in \Z^d} |k|^{2s}|\widehat{f}(k)|^2.
\]
When endowed with the norm $\|f\|_{H^s(\T^d)}^2= \|f\|_{L^2(\T^d)}^2+ \|f\|_{\dot{H}^s(\T^d)}^2$, $H^s(\T^d)$ is a Hilbert space.
\end{definition}

\begin{definition}\label{def:SobDef}
Let $1\leq r<\infty$ and $d\in\N$.  For $f\in L^r(\T^d)$ denote
\begin{itemize}
\item[i.] For $n\in\N$,
\[
\|f\|_{W^{n,r}(\T^d)}^r= \sum_{|\alpha|_1\le n} \|D^\alpha f \|_{L^r(\T^d)}^r.
\]
\item[ii.] For $0<s<1$,
\[
\|f\|_{\dot{W}^{s,r}(\T^d)}^r = \int_{\T^d} \int_{\T^d} \frac{|f(x+y)-f(x)|^r}{|y|^{d+sr}} dy dx.
\]
\end{itemize}
For $s>0$, the Sobolev--Slobodeckij space ${W}^{s,r}(\T^d)$ is the family of all functions  $f\in L^r(\T^d)$ for which the norm
\[
\|f\|_{W^{s,r}(\T^d)}^r=  \|f\|_{W^{\lfloor{s}\rfloor,r}(\T^d)}^r+ \sum_{|\alpha|_1=\lfloor{s}\rfloor} \|D^\alpha f\|_{\dot{W}^{\{s\},r}(\T^d)}^r,
\]
is finite.
\end{definition}
With this norm $W^{s,r}(\T^d)$ is a Banach space and a Hilbert space when $r=2$. In the latter case the spaces $ H^{s}(\T^d)$ and $W^{s,2}(\T^d)$ are equal and have equivalent norms (see e.g. \cite{BO}).

\begin{remark}
For an open subset $\Omega\subset \R^d$, $W^{s,r}(\Omega)$ is defined similarly, but with the semi--norm given by
\begin{align*}
	\|f\|_{\dot{W}^{s,r}(\Omega)}^r = \int_{\Omega} \int_{\Omega} \frac{|f(x)-f(y)|^r}{|x-y|^{d+sr}} dy dx,
\end{align*}
for $0<s<1$. Note that the main difference between this definition and the definition over the torus, is that Sobolev spaces on the torus
enforce smoothness across the boundary of $[-1/2,1/2]^d$ in its identification with $\T^d$, whereas this is not so on Euclidean domains.

\end{remark}

\subsection{Sobolev embeddings}\label{alphasection}

Given $d\in\N$, $s>0$ and $1\leq r<\infty$, the number $\alpha=s-d/r$ is an important quantitative measure of smoothness for the space $W^{s,r}(\T^d)$.
The next theorem shows that $\alpha$ determines Sobolev embeddings for $W^{s,r}(\T^d)$.
\begin{theorem}[Theorem 7.58 in \cite{Ad}] \label{thm:SobEmb}
Suppose that $X=\T^d$ or that $X=B$ is a ball in $\R^d$. Given
$0<s'\le s<\infty$ and $1< r\le  r'<\infty$, let $\alpha=s-d/r$ and $\alpha^\prime=s'-d/r'$.  If  $\alpha \geq \alpha^{\prime}$ then
\[ W^{s,r}(X) \subset W^{s',r'}(X),\]
and the embedding is continuous.
\end{theorem}

Theorem \ref{thm:SobEmb} is stated for $X=\R^d$ in \cite{Ad}, but it also extends to
more general domains such as balls $X=B$.
The case $X=\T^d$ in Theorem \ref{thm:SobEmb} follows from the case $X=\R^d$ since for $\Z^d$--periodic $f$, if $\psi \in C^{\infty}(\R^d)$ satisfies $\psi(x)=1$ for $x \in [-1,1)^d$, and $\text{supp}(\psi) \subset [-2,2)^d$, then
\begin{equation} \label{torus-from-euclidean}
\|f\|_{W^{s,r}(\T^d)}  \asymp \| \psi f \|_{W^{s,r}(\R^d)},
\end{equation}
where the implicit constants in \eqref{torus-from-euclidean} do not depend on $f$.
Due to the complexity of the seminorms in the definition of $W^{s,r}(\T^d)$, it is often easier to prove results for Sobolev spaces  under the additional assumption that $0<s\le 1$. Theorem \ref{thm:SobEmb} then allows to extend these results to cases where $s>1$.

 The next result embeds $W^{s,r}(\T^d)$ into the space of $\alpha$--H\"older continuous functions for $0<\alpha<1$, where as above $\alpha=s-d/r$. In fact, it implies that functions in  $W^{s,r}(\T^d)$ are slightly smoother then functions in the corresponding H\"older space, as it shows that for $f\in W^{s,r}(\T^d)$ we have the estimate
$|f(x)-f(y)| \lesssim  |x-y|^{\alpha} h(|x-y|),$ where $h(u) = o(1)$ as $u$ tends to $0$. The proof of this theorem follows closely the proof of Theorem 8.2 in \cite{DPV}.
\begin{theorem} \label{thm:FracHolderEmb}
For $1< r < \infty$ and $s>0$, let $\alpha=s-\frac{d}{r}$. If $0<\alpha<1$, then there exists a constant $C=C(d,s,r)$ such that for all $f \in W^{s,r}(\T^d)$ and for any $x, y \in \T^d$ we have,
\begin{equation} \label{holder-embed-thm-eq}
|f(x)-f(y)| \le C \| f \|_{\dot{W}^{s,r}(B_{2\tau}(z))} |x-y|^\alpha,
\end{equation}
where $z$ is the midpoint between $x$ and $y$, and $|x-y|=\tau$.
\end{theorem}

\begin{proof}

We first consider the case $0<s<1$. As mentioned above, in this case the result follows the proof of Theorem 8.2 in \cite{DPV}  with minor changes.  We give a sketch of the proof and leave the details to the reader.
For a measurable set $U$, let $\langle f \rangle_U:=\frac{1}{|U|}\int_Uf$ be the average of $f$ over $U$.  For any two points $x, y \in \T^d$ we have
\begin{align}\label{eqn:HolderBound541}
|f(x)-f(y)| \le |f(x)-\langle f\rangle_{B_\tau(x)}| + |\langle f\rangle_{B_\tau(x)}- \langle f\rangle_{B_\tau(y)}|+ |\langle f\rangle_{B_\tau(y)}-f(y)|.
\end{align}
The first and third terms in the right hand side can be bounded similarly.  Lemma 2.2 in \cite{Giusti} (see also Lemma
8.1 in \cite{DPV}) implies that if both $x$ and $y$ are Lebesgue points of $f$, then each of these terms can be bounded by $C\tau^{\alpha}\Phi(\tau)$ where,
\begin{align}\label{eqn:Giusti}
 \Phi(\tau):=\left( \sup_{0<\rho<\tau} \rho^{-sr}\int_{B_\rho(x)} |f(u)-\langle f
\rangle_{B_\rho(x)}|^rdu\right)^{1/r}.
\end{align}
Moreover, it follows from the proof of this estimate that the constant $C$ depends only on $s$, $d$, and $r$. Further, equation (8.3) of \cite{DPV} implies that $\Phi(\tau)\lesssim \sup_{0<\rho<\tau}\|f\|_{\dot{W}^{s,r}(B_{\rho}(x))} \le \|f\|_{\dot{W}^{s,r}(B_{2\tau}(z))}$
which provides the required bound.

Similarly, using H\"older's inequality and multiplying and dividing by $|u-v|^{d+sr}$, the middle term of \eqref{eqn:HolderBound541} may be bounded by a constant multiplying,
\begin{align}
\tau^{r\alpha} \int_{B_\tau(x)}\int_{B_\tau(y)} \frac{|f(u)-f(v)|^r}{|u-v|^{d+sr}} du dv \lesssim |x-y|^{r\alpha} \|f\|_{\dot{W}^{s,r}(B_{2\tau}(z))}^r,
\end{align}
where again the implied constants only depend on $s$, $d$, and $r$.  Therefore, the result holds when $0<s<1$.

The case $s\geq 1$ follows from a reduction to the case $0<s<1$.  To see this, note that since $0<s-\frac{d}{r}<1$, there exist $0<s'<1$ and $1< r' < \infty$ such that $s-\frac{d}{r}=s'-\frac{d}{r'}.$  Since Theorem \ref{thm:SobEmb} gives $W^{s,r}(\T^d) \subset W^{s',r'}(\T^d)$, the result now follows.
\end{proof}

Theorem \ref{thm:FracHolderEmb} implies that, whenever $0<\alpha<1$, if $f \in W^{s,r}(\T^d)$ satisfies $f(y)=0$ for some $y\in\T^d$, then for $\tau>0$,
\begin{align} \|f\|_{L^r(B_\tau(y))} \le C \tau^{s} \| f \|_{\dot{W}^{s,r}( B_{2\tau}(y))}. \label{eqn:holderlocalization}\end{align}
The work in \cite{JL} and \cite{SC} shows that similar bounds as in \eqref{eqn:holderlocalization} hold also in the case where $\alpha <0$, if the Hausdorff dimension of the zero set of the function is large enough.

\begin{theorem}[\cite{JL} and \cite{SC}]\label{thm:PI}
Fix $1<r<\infty$ and let $d$, $\sigma$, and $s$ satisfy, $0\le (d-\sigma)/r<s\le 1$.

If $f \in W^{s,r}(\T^d)$ satisfies $H^{\sigma}(\Sigma(f))>0$ then there exist a constant $C>0$, and a  closed set $T\subset \Sigma(f)$ with $H^{\sigma}(T)>0$ such that the following holds. For every $\epsilon>0$ one can find a collection of pairwise disjoint balls $\mathcal{B}=\{B_k\}_{k=1}^\infty= \{B_{\tau_k}(x_k)\}_{k=1}^\infty$ with $x_k \in T$ such that
\begin{itemize}
\item[(i)]  $T\subset \bigcup_{k=1}^\infty B_{5\tau_k}(x_k)$.
\item[(ii)] $|\bigcup_{k=1}^\infty B_k|<\epsilon$.
\item[(iii)] $ \|f\|_{L^r(B_{k})} \le C {\tau_k}^s \|f\|_{\dot{W}^{s,r}(B_{k})}\qquad\forall k\in\N.$
\end{itemize}

\end{theorem}

\begin{remark}\label{rem:PoincareIneq-s}
	Theorem \ref{thm:PI} also holds for $f \in W^{s,r}(\T^d)$ when $s>1$ with the condition
	\begin{equation*} \frac{d}{r}-\frac{\sigma}{d-\sigma} < s < \frac{d}{r}+1\end{equation*}
	replacing the condition $(d-\sigma)/r<s\le 1$.  The upper bound shows that there exists an $r'>r$ such that $s-d/r=1-d/r'$, and so Theorem \ref{thm:SobEmb} implies that $f \in W^{1,r'}(\T^d)$.  Rearranging the lower bound and applying the previous equality gives
		${d}/{r'} < 1+\sigma/(d-\sigma)= d/({d-\sigma})$,
	which is equivalent to  $(d-\sigma)/r' <1$.  Thus, Theorem \ref{thm:PI} can be applied to the space $W^{1,r'}(\T^d)$.  Clearly, (i) and (ii) continue to hold. To see that (iii) still holds, H\"older's inequality gives,
	\[\|f\|_{L^r(B_{k})} \lesssim \tau^{d/r-d/r'} \|f\|_{L^{r'}(B_{k})} \lesssim \tau^{1+d/r-d/r'} \|f\|_{\dot{W}^{1,r'}(B_{k})}\lesssim \tau^s \|f\|_{\dot{W}^{s,r}(B_{k})}. \]
\end{remark}

\begin{proof}[Proof of Theorem \ref{thm:PI}]
The proof is a combination of results from \cite{JL} and \cite{SC}.
Theorem 2.1 of \cite{JL} and Theorem 1.3 of \cite{SC} provide sufficient conditions on $T$ and $B_{k}$ for condition (iii) to hold in the case that $s\le d/r$, and equation \eqref{eqn:holderlocalization} provides the same for $s>d/r$. The proof of Theorem 1.1 in \cite{JL} shows how to construct $T$ and choose $\mathcal{B}$ so that these sufficient conditions are met for each ball in $\mathcal{B}$ and such that the conditions (i) and (ii) are also satisfied.

\end{proof}

\subsection{Restrictions of Sobolev functions over lines} \label{subsec:SobLines}

For $x=(x_1,...,x_{d-1}) \in \R^{d-1}$ and $1\leq i \leq d$ denote by $L_i(x)$ the line parallel to the $i$th coordinate axis passing through the point $(x_1,...,x_{i-1},0,x_{i},...,x_{d-1})$, that is,  $L_i(x)=\{(x_1,...,x_{i-1},t,x_{i},...,x_{d-1})\}$.
An equivalent, definition to the  Sobolev--Slobodeckij  spaces is given in the following proposition.
\begin{proposition}[e.g. \cite{BBM, Ad}]\label{sobolevonlines}
Fix $d\in \N$, $0<s\le 1$ and $1\leq r<\infty$. Then, for $f\in W^{s,r}(\T^d)$ we have,
\begin{align}\label{eqn:AltSobNormLines}
\|f\|^r_{\dot{W}^{s,r}(\T^d)} \asymp
	 \sum_{k=1}^d \int_{\T^{d-1}} \|f|_{L_i(x)}\|^r_{\dot{W}^{s,r}(\T)} dx,
\end{align}
where the implied constants do not depend on $f$.
\end{proposition}

Theorems \ref{thm:SobEmb} and \ref{thm:FracHolderEmb} imply that if $\alpha=s-\frac{d}{r}>0$ then every $f\in W^{s,r}(\T^d)$
has a H\"older continuous representative.  This is not necessarily true when $\alpha\leq 0$, but the next result shows that for non-negative functions there exist representatives which are H\"older continuous on almost every line parallel to a coordinate axis, and their zero set is exactly equal to the set $\Sigma(f)$ defined in \eqref{eqn:ZeroSetDef}.

For a non-negative function $f \in W^{s,r}(\T^d)$ denote
\begin{equation} \label{fstar-defs}
f^*(x)= \limsup_{\tau\rightarrow 0}f_{\tau}(x) \ \ \ \ \ \hbox{; } \ \ \ \ \ f_{\tau}(x)= \frac{1}{|I_\tau(x)|} \int_{I_\tau(x)} f(u) du,
\end{equation}
and note that by the Lebesgue differentiation theorem $f^*$ is equal to $f$ almost everywhere.

\begin{proposition}\label{prop:HolderEmbedLines}
Let $s>0$, $d\in \N$, and $1<r<\infty$ be such that $0<s-1/r\le 1$.  Then, for a nonnegative function $f \in W^{s,r}(\T^d)$, the representative $f^*$ of $f$
satisfies $f^*|_L$ is $(s-1/r)$--H\"older continuous on almost every line $L$ in $\T^d$ which is parallel to an axis, and its zero set is exactly $\Sigma(f)$.
\end{proposition}

\begin{proof}
 It suffices to prove the proposition for lines parallel to the first coordinate axis, therefore for fixed $y_0\in\T^{d-1}$ we consider the function of one variable $f(x,y_0)$, with $x\in\T$. By (\ref{eqn:AltSobNormLines}) we have $f(x,y_0) \in W^{s,r}(\T)$ for almost every $y_0 \in\T^{d-1}$, in what follows we assume $y_0$ to satisfy this condition.
It follows from Theorem \ref{thm:FracHolderEmb} that for almost every $x_1,x_2\in\T$,
\[
|f(x_1,y_0)-f(x_2,y_0)|\leq C\|f(\cdot,y_0)\|_{\dot{W}^{s,r}(\T)}|x_1-x_2|^{s-1/r},
\]
where  $C=C(d,s,r)$. Therefore, for $\tau>0$, the function $\phi^{\tau} (x,y_0)=(2\tau)^{-1}\int_{x-\tau}^{x+\tau} f(t,y_0) dt$ satisfies the same inequality everywhere on $\T$, as can be checked by noting that $\phi^{\tau} (x,y)=(2\tau)^{-1}\int_{-\tau}^{\tau} f(t+x,y) dt$. Next, since
\[
|f_\tau(x_1,y_0)-f_\tau(x_2,y_0)| \le (2\tau)^{-(d-1)} \int_{I_\tau(y_0)} \left|\phi^{\tau}  (x_1,z)- \phi^{\tau}  (x_2,z) \right|dz,
\]
we conclude that
\begin{align*}
|f^*(x_1,y_0)-f^*(x_2,y_0)| &\le\limsup_{\tau\rightarrow 0} |f_\tau(x_1,y_0)-f_\tau(x_2,y_0)|\\
&\le C\limsup_{\tau\rightarrow 0}(2\tau)^{-(d-1)} \int_{I_\tau(y_0)} \|f(\cdot,z)\|_{\dot{W}^{s,r}(\T)}dz\cdot |x_1-x_2|^{s-1/r}.
\end{align*}
Now, \eqref{eqn:AltSobNormLines} implies that the function $G(z)= \|f(\cdot,z)\|_{\dot{W}^{s,r}(\T)} \in L^1(\T^{d-1})$ and therefore, by the Lebesgue Differentiation Theorem the last limit is finite for almost every $y_0\in\T^{d-1}$. The result follows.
\end{proof}

\subsection{Anisotropic Bessel potential spaces}\label{subsection-mixed}

Next we define anisotropic Bessel potential spaces and prove an embedding similar to Theorem \ref{thm:FracHolderEmb} for these spaces.  Similar spaces were studied in Chapter 5 of \cite{Triebel} and the references therein.

For  $\vec{s}=(s_1,...,s_d)\in (0,\infty)^d$ the anisotropic Bessel potential space, $H^{\vec{s}}(\T^d)$, is defined as the space consisting of all functions $f \in L^2(\T^d)$ for which the following
semi--norm is finite
\begin{align}
	\|f\|_{\dot{H}^{\vec{s}}(\T^d)}^2= \sum_{k \in \Z^d}\left( |k_1|^{2s_1}+\cdots +|k_d|^{2s_d}\right) |\widehat{f}(k)|^2.
\end{align}
When endowed with the norm $\|f\|_{H^{\vec{s}}(\T^d)}^2= \|f\|_{L^2(\T^d)}^2+ \|f\|_{\dot{H}^{\vec{s}}(\T^d)}^2$,  $H^{\vec{s}}(\T^d)$ is a Hilbert space.

\begin{lemma}\label{lem:nonsymCont}
Suppose $\vec{s}$ satisfies $\sum_{j=1}^d \frac{1}{s_j} < 2$.  If $f \in H^{\vec{s}}(\T^d)$ then $\widehat{f} \in \ell^1(\Z^d)$.  In particular, $f$ is continuous.
\end{lemma}

\begin{proof}
We first note that
\begin{equation}
S = \sum_{k \in \Z^d}\frac{1}{1+|k_1|^{2s_1}+\cdots +|k_d|^{2s_d}}<\infty.
\end{equation}
Indeed, estimating the sum $S$ by an integral, and making an appropriate change of variables, we find that

\[
S \lesssim \int_{[0, \infty]^d} \frac{  \prod_{j=1}^d  t_j^{\frac{1}{s_j}-1} }{1+ t_1^2 + \cdots +t_d^2} d\vec{t}.
\]
Converting to spherical coordinates, and integrating away the angular terms, gives
\[
S \lesssim \int_{0}^\infty \frac{ r^{\left( \sum_{j=1}^d \frac{1}{s_j} \right) - d}}{1+ r^2} r^{d-1} dr,
\]
which is finite as $\sum_{j=1}^d \frac{1}{s_j} < 2$. The lemma now follows from a standard application of the Cauchy Schwartz inequality.
\end{proof}

Next, we show that the conditions of Lemma \ref{lem:nonsymCont} imply that a function in $ H^{\vec{s}}(\T^d)$ is not only continuous, but in fact satisfies a stronger mixed--H\"older continuity property. For this goal, given $\vec{s}\in (0,\infty)^d$, denote
\begin{equation} \label{eqn:alphaell}
\alpha_{\ell}=s_{\ell}\left(1-\frac{1}{2}\sum_{j=1}^d \frac{1}{s_j}\right), \qquad \ell=1,...,d.
\end{equation}
The quantities $\alpha_{\ell}$ play a similar role as $\alpha$ from Section \ref{alphasection}.
In particular, if $s_{\ell}=s$ for all $\ell$, and $r=2$, then  $\alpha_{\ell}=s-\frac{d}{2}=\alpha$.

We are now ready to prove an anisotropic analogue of Theorem \ref{thm:FracHolderEmb}.

\begin{theorem}\label{thm:nonsymSobolevEmb}
Let $G \in H^{\vec{s}}(\T^d)$ where $\sum_{j=1}^d \frac{1}{s_j}<2$.  For $\ell=1,...,d$, let $\alpha_{\ell}$ be as in \eqref{eqn:alphaell}, and assume that $0<\alpha_{\ell}<1$. Then, there exists a constant $C$ such that
	if $x_1,x_2\in\T^d$ are equal in all but their $\ell$'th coordinate then
\begin{align} |G(x_1)- G(x_2)| \le C R(|x_1^{(\ell)}-x_2^{(\ell)}|) |x_1^{(\ell)}-x_2^{(\ell)}|^{\alpha_{\ell}},\nonumber \end{align}
where $\lim_{\tau\rightarrow 0} R(\tau)=0$, and $x_j^{(\ell)}$ is the $\ell$'th coordinate of $x_j$, $j=1,2$.
\end{theorem}

\begin{proof}
To avoid cumbersome notations, we present the proof in the case of two variables. The general case can be proved in much the same way, we leave the details to the reader. Without loss of generality, we may assume that $\ell=1$.

Fix $N\in\N$. Since $\sum_{j=1}^2 \frac{1}{s_j}<2$, the Fourier series of $G$ converges absolutely and so for $x_1,x_2,y\in\T$ and $\tau=|x_2-x_1|$, we may apply the Cauchy Schwartz inequality and find that,
\begin{align}\label{sigma}
	|G(x_1,y)-G(x_2,y)|\lesssim  \sum_{(k,n)\in \Z^2} |\widehat{G}(k,n)| | \sin(\pi \tau k)| \leq L_1^{\frac{1}{2}}+L_2^{\frac{1}{2}},
\end{align}
where
\[
L_1=\left(\sum_{(k,n)\in \Z^2; |k|\leq N} ({|k|^{2s_1}+|n|^{2s_2}})|\widehat{G}(k,n)|^2\right)\left(\sum_{(k,n)\in \Z^2; |k|\leq N} \frac{| \sin(\pi \tau k)|^2}{|k|^{2s_1}+|n|^{2s_2}}\right),
\]
and $L_2$ is the similar expression with the restriction $|k|>N$ replacing that of $|k|\leq N$.

To estimate $L_1$ we first note that for a fixed $k\neq 0$, since $2s_2>1$, we have
\begin{align}\label{sumestimate}
\sum_{n \in \Z} \frac{1}{|k|^{2s_1}+|n|^{2s_2}}&\lesssim \int_{0}^\infty\frac{dt}{|k|^{2s_1}+|t|^{2s_2}}\lesssim \int_0^{\infty}\frac{dt}{\big(|k|^{\frac{s_1}{s_2}}+t\big)^{2s_2}}\lesssim |k|^{-2s_1+\frac{s_1}{s_2}}.
\end{align}
Combining with the definition of the anisotropic Bessel norm, and the estimate $|\sin \theta|\leq|\theta|$, we therefore obtain,
\[
\begin{aligned}
L_1&\lesssim \|G\|_{\dot{H}^{\vec{s}}(\T^d)}^2 \tau^2 \sum_{k =1}^Nk^{2-2s_1+\frac{s_1}{s_2}}\\
&\lesssim \|G\|_{\dot{H}^{\vec{s}}(\T^d)}^2 \tau^2 N^{2-2\alpha_1}.
\end{aligned}
\]

We turn to the estimate of $L_2$. First, we denote
\[
\Phi(N)=\left(\sum_{(k,n)\in \Z^2; |k|> N} ({|k|^{2s_1}+|n|^{2s_2}})|\widehat{G}(k,n)|^2\right)^{\frac{1}{2}}.
\]
Applying (\ref{sumestimate}) once again we obtain,

\[
L_2\lesssim \Phi(N)^2\sum_{|k|\geq |N|}|k|^{-2s_1+\frac{s_1}{s_2}}\lesssim  \Phi(N)^2|N|^{1-2s_1+\frac{s_1}{s_2}}\lesssim \Phi(N)^2|N|^{-2\alpha_1}.
\]
Plugging these estimates into (\ref{sigma}) we therefore find that,
\[
\begin{aligned}
|G(x_1,y)-G(x_2,y)|&\lesssim   \tau N^{1-\alpha_1}+\Phi(N)|N|^{-\alpha_1}\\&=  N^{1-\alpha_1}\left( \tau + \Phi(N)/N\right),
\end{aligned}
\]
with the implied constant depends on $\|G\|_{\dot{H}^{\vec{s}}}$, but not on $x_1,x_2,y$.

If $\Phi(N)$ is equal to zero for some $N$, then $G$ is a trigonometric polynomial in the first variable, and the claim trivially follows.  We can therefore assume that $\Phi(N)$ is different from zero for every $N$, and we note that, since $G \in H^{\vec{s}}(\T^d)$, this function is decreasing to zero as $N$ tends to $\infty$. For small enough $\tau$, choose $N=N(\tau)$ to be the integer satisfying $\Phi(N)/N\leq \tau<\Phi(N-1)/(N-1)$. Note that as $\tau$ tends to zero, the corresponding $N(\tau)$ tends to infinity, and therefore $\Phi(N(\tau)-1)$ tends to zero. Plugging $N=N(\tau)$ into the above estimate we get
\[
\begin{aligned}
|G(x_1,y)-G(x_2,y)|&\lesssim  N^{1-\alpha_1}\tau=\tau^{\alpha_1}(N\tau)^{1-\alpha_1}\\
&\leq \tau^{\alpha_1}(2\Phi(N-1))^{1-\alpha_1}.
\end{aligned}
\]
As $\Phi(N(\tau)-1)$ tends to zero when $\tau$ tends to zero, and $\alpha_1<1$, the result follows.

\end{proof}

\subsection{Sobolev spaces over \texorpdfstring{$\R^d$}{Rd}}
Most of the results listed above hold when the Sobolev spaces over $\T^d$ are replaced by their analog over $\R^d$.  For $0<s<1$, the seminorm for $W^{s,r}(\R^d)$ is defined by
\begin{align*}
\|f\|_{\dot{W}^{s,r}(\R^d)}^r = \int_{\R^d} \int_{\R^d} \frac{|f(x)-f(y)|^r}{|x-y|^{d+sr}} dy dx,
\end{align*}
and the space $W^{s,r}(\R^d)$ is defined analogously to Definition \ref{def:SobDef}.  Theorems \ref{thm:SobEmb}, \ref{thm:FracHolderEmb}, and \ref{thm:PI}, and Proposition \ref{sobolevonlines} all hold when $\T^d$ is replaced by $\R^d$.

The anisotropic Bessel potential spaces can also be considered over $\R^d$ instead of $\T^d$ (See \cite{Triebel}).  With a nearly identical proof, it can be shown that Theorem \ref{thm:nonsymSobolevEmb} holds when $H^{\vec{s}}(\R^d)$ replaces $H^{\vec{s}}(\T^d)$.

\subsection{Auxiliary constructions.}

In this subsection we construct some examples of functions which belong to certain Sobolev spaces. These functions will be used in our proofs for different sharpness statements.

\subsubsection{} We start with the following construction. Let $\eta\in C^\infty(\R^d)$ be a nonnegative function satisfying $\eta(x)=1$ for $x \in B_{1/8}(0)$, $\eta(x)=0$ for $x \notin B_{1/4}(0)$, and $\eta(x)$ positive for $x \in B_{1/4}(0)$.
For $\beta>0$, let $w_\beta$ be the $\Z^d$--periodic function defined by $w_\beta(x)=(1-\eta(x))+\eta(x)|x|^\beta$ for $x \in [-1/2, 1/2)^d$.

\begin{proposition}\label{prop:MultCounter}
Fix $0<\beta<1$. Then,
\begin{itemize}
	\item[i.] $w_\beta(0)=0$, and $w$ has no other zero in $\T^d$.
	\item[ii.] $u_\beta=1/w_\beta \in L^{\frac{2q}{q-2}}(\T^d) \subset \mathcal{M}_{2}^q$ whenever $q>2d/(d-2\beta)$.
	\item[iii.] $w_\beta \in W^{s,r}(\T^d)$ for any $s$ and $r$ satisfying $0<s\le d$, $1<r <\infty$, and $s-\frac{d}{r}<\beta$.
\end{itemize}
\end{proposition}

\begin{proof}
Part (i) is clearly true by the definition of $w_\beta$.  For part (ii), it suffices to consider $w_\beta$ in $B_{1/8}(0)$, as the function is bounded away from zero outside of this region.  We have,
\begin{align*}
	\int_{B_{1/8}(0)} |w_\beta(x)|^{\frac{-2q}{q-2}} dx = \int_{B_{1/8}(0)} |x|^{-\frac{2q\beta}{q-2}} dx,
\end{align*}
and this is finite if and only if $\frac{2q\beta}{q-2} <d$, which is equivalent to $q>2d/(d-2\beta)$.  The inclusion, $L^{\frac{2q}{q-2}}(\T^d) \subset \mathcal{M}_{2}^q$ follows from the discussion after Theorem \ref{thm:ScalarValuedLargeZeroResult}.

We turn to part (iii). It is enough to prove it for $s=d$, as Theorem \ref{thm:SobEmb} would then imply that the result holds for all $0<s\leq d$.
Note that away from $0$, $w_\beta$ is smooth, and so it suffices to consider a neighborhood of $0$.
Let $D$ denote some $d^{th}$ order partial derivative operator (in the sense of distribution).  A straightforward induction argument shows that $|Dw_\beta(x)|\lesssim |x|^{\beta-d}$ for $x \in B_{1/8}(0)$, and thus,
\[ \int_{B_{1/8}(0)} |D w_\beta (x)|^r dx \lesssim \int_{B_{1/8}(0)} |x|^{(\beta-d)r} dx<\infty, \]
whenever $(\beta-d)r >-d$.  This is equivalent to $d-\frac{d}{r} < \beta$.
\end{proof}

\subsubsection{}  Fix $0< \beta<1$, denote $h_{\beta} \in L^2(\R)$ by
\begin{displaymath}
   h_{\beta}(x) = \left\{
     \begin{array}{lr}
       0 & : x \ge \frac{1}{2}\\
       (\frac{1}{2}-|x|)^{\beta/2} & : -\frac{1}{2}\le x \le \frac{1}{2}\\
     \end{array}
   \right.
\end{displaymath}
We have the following.

\begin{proposition}\label{prop:CqCounterD1}
We have $h_{\beta} \in W^{s,2}(\R)$ for all $s<\frac{1+\beta}{2}$.
\end{proposition}

\begin{proof}
 We will use the fact that as sets $W^{s,2}(\R)=H^s(\R)$. Let $f_{\beta}= \widehat{h_\beta}$.  A direct calculation shows
\begin{align}
	f_{\beta} (\xi)&= 2\int_{0}^{\frac{1}{2}} \cos(2\pi x \xi) (\frac{1}{2}-x)^{\beta/2} dx  = \xi^{-1-\beta/2}2^{-\beta/2}\int_{0}^{\xi} \cos(\pi y) (\xi-y)^{\beta/2} dy  \label{eqn:hello}
\end{align}
Now we show that $|\int_{0}^{\xi} \cos(\pi y) (\xi-y)^{\beta/2} dy |$ is bounded by a constant for all $\xi>1$.  First,
\begin{align}
	|\int_{\xi-1}^{\xi} \cos(\pi y) (\xi-y)^{\beta/2} dy|&\le 1 \label{eqn:IDontKnow}
\end{align}

Second, after two steps of integration by parts, we find,
\begin{align}
	\left| \int_{0}^{\xi-1} \cos(\pi y) (\xi-y)^{\beta/2} dy\right|
	\lesssim 1+\left|\int_0^{\xi-1} \cos(\pi y) (\xi-y)^{\beta/2-2}dy\right| \lesssim 1, \label{eqn:twentyseven}
\end{align}
where the implied constants only depend on $\beta$.
It follows from equations \eqref{eqn:hello}, \eqref{eqn:IDontKnow}, and  \eqref{eqn:twentyseven} that for any $|\xi|>1$, $|f_\beta(\xi)|^2 \lesssim |\xi|^{-2-\beta}$ and so
\begin{align*}
	\int_{\R} |\xi|^{2s} |\widehat{h_\beta}(\xi)|^2 d\xi=\int_{\R} |\xi|^{2s} |f_\beta(\xi)|^2 d\xi <\infty
\end{align*}
whenever $ s<\frac{1+\beta}{2}$.

\end{proof}

\section{Proofs for the multiplier results}\label{sec:Proofs}

\subsection{A proof for Theorem \ref{thm:SobolevResultSingleZero-2-2}}

Here we prove the following more general version of Theorem \ref{thm:SobolevResultSingleZero-2-2}.

\begin{theorem}[Nitzan, Northington, Powell]\label{thm:SobolevResultSingleZero-r-2}
 Let $1<r<\infty$ and $\frac{d}{r}<s \le d(1/2+1/r)$. Suppose $w \in W^{s,r}(\T^d)$ and $w$ has a zero.
\begin{itemize}
	\item[i.] If $s<\frac{d}{r}+1$, then $u=\frac{1}{w} \notin \mathcal{M}_2^q$ for any $q$ satisfying $2 \le q \le \frac{d}{d(1/2+1/r)-s}$.  Conversely, for any $0<s\le d$ and $q>\frac{d}{d(1/2+1/r)-s}$, there exists $w \in W^{s,r}(\T^d)$ such that $w$ has a zero and $u=\frac{1}{w}\in \mathcal{M}_2^q$.
	\item[ii.] If $s=\frac{d}{r}+1$, then $u=\frac{1}{w} \notin \mathcal{M}_2^q$ for any $q$ satisfying $2 \le q<\frac{2d}{d-2}$.
	\end{itemize}
\end{theorem}

\begin{proof}
To prove the implication in part (i), suppose for contradiction that $u=\frac{1}{w} \in \mathcal{M}_2^q$ for a function $w$ satisfying the condition of the theorem and $q \le \frac{d}{d(1/2+1/r)-s}$. In particular, we will assume without loss of generality that $w(0)=0$.

Let $\alpha=s- \frac{d}{r}$ be the quantity discussed in Section \ref{alphasection}, and note that the condition on $q$ and $s$ implies that,
\begin{equation}\label{alphaisbig}
d(\frac{1}{2}-\frac{1}{q})\leq \alpha <1.
\end{equation}

 The $(2,q)$--multiplier condition on $u=\frac{1}{w}$ can be rewritten as follows: For any function $g$ over $\T^d$ which satisfies $wg\in L^2(\T^d)$ we have:
\begin{equation} \label{eqn:BFI}\|\widehat{g}\|_{\ell^q(\Z^d)} \le C \|wg\|_{L^2(\T^d)}.\end{equation}

For $\tau>0$, denote $I_\tau=I_\tau(0)$.  Substituting $g=\chi_{I_\tau}$ into \eqref{eqn:BFI} and using standard bounds on the $\ell^q$ norm of its Fourier transform, we have
\begin{equation}\label{eqn:local} \tau^{d(1-1/q)}\lesssim \|\widehat{\chi_{I_\tau}}\|_{\ell^q(\Z^d)} \lesssim \|w\|_{L^2(I_\tau)},  \end{equation}
with constants not depending on $\tau$.

To estimate $ \|w\|_{L^2(I_\tau)}$, we use the fact that $w(0)=0$,  \eqref{holder-embed-thm-eq}, and (\ref{alphaisbig}). We get,
\begin{align}
\int_{I_\tau} |w(x)|^2 dx \lesssim \|w\|^2_{\dot{W}^{s,r}(B_{2\sqrt{d}\tau})}|I_{\tau}| \tau^{2\alpha}\lesssim \|w\|^2_{\dot{W}^{s,r}(B_{2\sqrt{d}\tau})}\tau^{2d(1-1/q)}.
\end{align}
Combining this estimate with \eqref{eqn:local}, we find that
\begin{equation}
	1 \lesssim  \|w\|_{\dot{W}^{s,r}(B_{2\sqrt{d}\tau})}, \end{equation}
which is absurd as $\|w\|_{\dot{W}^{s,r}(B_{2\sqrt{d}\tau})}$ tends to zero when $\tau$ tends to zero.

To prove part (ii) we recall that if $w \in W^{d/r+1,r}(\T^d)$, then $w\in W^{s,r}$ for all $s<d/r+1$.  Thus, we can apply part (i) to find that $u \notin \mathcal{M}_2^q$ for any $q<\frac{d}{d(1/2+1/r)-(d/r+1)}=\frac{2d}{d-2} $.

The sharpness in part (i) follows from Proposition \ref{prop:MultCounter}.  Indeed, for $\epsilon>0$ let $\beta=s-d/r+\epsilon$ and $w_\beta$ be as in \ref{prop:MultCounter}.  Then we have $w_\beta \in W^{s,r}(\T^d)$, $w_\beta(0)=0$, and $u_\beta=1/w_\beta \in\mathcal{M}_2^q$ for $q>d/(d(1/2+1/r)-s-\epsilon)$. This completes the proof.
\end{proof}

\begin{remark}
Note that due to the local nature of the proof above, Theorem \ref{thm:SobolevResultSingleZero-r-2} holds also when the condition $w\in W^{s,r}(\T^d)$ is replaced by the condition $w\in W^{s,r}(B)$ for some ball $B \subset \T^d$ containing a zero of $w$. Similar local versions hold for all the results regarding multipliers appearing in the paper.
\end{remark}

\subsection{A proof for Theorem \ref{thm:ScalarValuedLargeZeroResult}}

Here we prove the following more general version of Theorem \ref{thm:ScalarValuedLargeZeroResult}, part (ii) of which follows essentially from the main results in \cite{JL,SC}.

\begin{theorem}[Nitzan, Northington, Powell]\label{thm:ScalarValuedLargeZeroGeneralResult}
 Let $2\le r<\infty$, $0< \sigma < d$, and \[d/r-\sigma/\max(r,d-\sigma)< s\le  (d-\sigma)(1/2+1/r).\] Suppose $w \in W^{s,r}(\T^d)$ and $\mathcal{H}^\sigma(\Sigma(w))>0$
\begin{itemize}
	\item[i.] If $ s < d/r+1 $, then $u=\frac{1}{w} \notin \mathcal{M}_2^q$ for any $q$ satisfying \[2 \le q \le \frac{d}{d(1/2+1/r)-\sigma/r-s}.\]
	\item[ii.] If $s=(d-\sigma)(1/2+1/r)<d/r+1$, then $u=\frac{1}{w} \notin L^2(\T^d)$, and thus $u=\frac{1}{w} \notin \mathcal{M}_2^q$ for any $q$.
	\item[iii.] If $s=d/r+1 \le (d-\sigma)(1/2+1/r)$, then $u=\frac{1}{w}\notin \mathcal{M}_2^q$ for and $q$ satisfying $2 \le q < \frac{d}{d/2-\sigma/r-1}$.
	\end{itemize}
\end{theorem}

\begin{proof}

We first note that part (ii) of this theorem follows from the main result in \cite{JL} and \cite{SC}. Indeed, in the case that $s=(d-\sigma)(\frac{1}{2}+\frac{1}{r}) \le 1$, it is stated there that if $u=1/w\in L^2(\T^d)$ and $w \in W^{s,r}(\T^d)$ then $\mathcal{H}^{\sigma}(\Sigma(w))=0$. Since for every value of $2\leq q$ we have $\mathcal{M}_2^q \subset \mathcal{M}_2^{\infty}= L^2(\T^d)$, part (ii) follows.  Now, for the case $s>1$, we point out that an application of Theorem \ref{thm:PI}, as in the proof of Remark \ref{rem:PoincareIneq-s}, allows one to easily extended the main result of \cite{JL} and \cite{SC} to the claimed range of parameters.

To prove part (i), suppose for a contradiction that $u=\frac{1}{w} \in \mathcal{M}_2^q$ for a function $w$ satisfying the conditions of the theorem. Note that the condition on $q$ implies that
\begin{equation}\label{ohdear}
\sigma \ge r\left[d(\frac{1}{2}+\frac{1}{r}-\frac{1}{q})-s\right].
\end{equation}
Since $d/r-\sigma/\max(r,d-\sigma)<s<d/r+1$, we can apply Theorem \ref{thm:PI}  and the remark following it, to obtain a set $T$ with $\mathcal{H}^\sigma(T)>0$, and a constant $C$, which satisfy the conditions of the theorem.

For $\epsilon>0$ let $\mathcal{B}^{(\epsilon)}=\{B_k\}_{k=1}^\infty= \{B_{\tau_k}(x_k)\}_{k=1}^\infty$ be the collection of disjoint balls guaranteed by Theorem \ref{thm:PI} and denote $V^{(\epsilon)}= \bigcup_{k=1}^\infty B_k$ (we suppress the superscript "$\epsilon$" from the balls to avoid cumbersome notations).  For ${B_k=B_{\tau_k}(x_k)} \in \mathcal{B}$, let $I_k = I_{\tau_k/\sqrt{d}}(x_k)$ so that $I_k \subset B_k$. As in the proof of Theorem \ref{thm:SobolevResultSingleZero-r-2}, the Fourier multiplier property implies that  for each such ball $B_k$ we have,
\[
	\tau_k^{d(1-\frac{1}{q})} \lesssim \| \widehat{ \chi_{I_{k}}} \|_{\ell^q(\Z^d)}
		\lesssim \|w\|_{L^{2}(I_k)}\lesssim\|w\|_{L^{2}(B_k)}.
\]
H\"olders inequality and part (iii) of Theorem \ref{thm:PI} now imply that
\[
	\tau_k^{d(1-\frac{1}{q})} \lesssim \tau_k^{\frac{d(r-2)}{2r}} \|w\|_{L^r(B_k)}
		\lesssim \tau_k^{s+\frac{d(r-2)}{2r}} \|w\|_{\dot{W}^{s,r}(B_k)}.
\]

Rearranging this inequality and plugging it in (\ref{ohdear}), we find that
\begin{equation}\label{eqn:local2}
\tau_k^\sigma\le \tau_k^{r\left[d(\frac{1}{2}+\frac{1}{r}-\frac{1}{q})-s\right]} \lesssim \|w\|^r_{\dot{W}^{s,r}(B_k)}.
\end{equation}
It follows that,
\[\sum_k (\tau_{k}^{(\epsilon)})^\sigma\lesssim \sum_k\|w\|^r_{\dot{W}^{s,r}(B_k)}= \|w\|_{\dot{W}^{s,r}(V^{(\epsilon)})}^r,\]
and so, due to part (ii) of  Theorem \ref{thm:PI}
\[ \lim_{\epsilon \rightarrow 0} \sum_k (\tau^{(\epsilon)}_{k})^\sigma = 0.\]
This, combined with part (i) of  Theorem \ref{thm:PI}, contradicts the fact that $\mathcal{H}^\sigma(T)>0$. Part (i) follows.

For part (iii), since $w \in W^{d/r+1,r}(\T^d)$, we also have $w\in W^{s,r}$ for all $s<d/r+1$.  Part (i) shows that $u \notin \mathcal{M}_2^q$ for any $q<\frac{d}{d(1/2+1/r)-\sigma/r-(d/r+1)}=\frac{d}{d/2-\sigma/r-1} $.
\end{proof}

\subsection{The anisotropic case}

Next, we formulate and prove a version of our results for functions in the anisotropic Bessel potential spaces $H^{\vec{s}}(\T^d)$, which were defined in Section \ref{subsection-mixed}. Given a $d$--tuple, $\{s_j\}_{j=1}^d$, we denote
$\ell(\vec{s})=\sum_{j=1}^d\frac{1}{s_j}$ and recall the notation $\alpha_j=s_j(1-\frac{1}{2}\ell(\vec{s}))$.

\begin{theorem}[Nitzan, Northington, Powell]\label{thm:nonsymMult}
Let $\{s_j\}_{j=1}^d$ be such that $0<s_j$ and $0<\alpha_j<1$ for every $j=1,2,...,d$.  Suppose $w \in H^{\vec{s}}(\T^d)$  and $w$ has a zero, then  $u=\frac{1}{w} \notin \mathcal{M}_2^q$ for any $q$ satisfying $2\leq q\leq \ell(\vec{s})/(\ell(\vec{s})-1)$.
\end{theorem}

\begin{proof}[Proof of Theorem \ref{thm:nonsymMult}]
Suppose for a contradiction that $u=\frac{1}{w} \in \mathcal{M}_2^q$ for some function $w$ satisfying
the conditions of the theorem.  Without loss of generality, we may assume that $w(0)=0$, and that $s_{\text{min}}=s_1\le s_2\le \cdots \le s_d$.

Since $0<\alpha_j<1$ for every $j=1,2,...,d$ we have, in particular, that $\ell(\vec{s})<2$ and therefore, by Lemma \ref{lem:nonsymCont}, that $w$ is continuous.
 Moreover, by Theorem \ref{thm:nonsymSobolevEmb}, we have
\begin{equation}\label{mixed holder}
	|w(x)|=|w(x)-w(0)|\le R(|x|) \Sigma_{j=1}^d|x_j|^{\alpha_j},
\end{equation}
where $\lim_{\tau \rightarrow 0} R(\tau) = 0$.

For $\tau>0$, let $I_{\tau,\vec{s}}$ be the rectangle defined by $I_{\tau,\vec{s}}=\Pi_{j=1}^d[-\tau_j/2,\tau_j/2]$, where $\tau_j= \tau^{s_1/s_j}$.  The inequality in (\ref{mixed holder}) implies that
\[
	\int_{I_{\tau,\vec{s}}} |w(x)|^2 dx \le \max_{x \in I_{\tau,\vec{s}}} (R(|x|))^2 \sum_{j=1}^d \int_{I_{\tau,\vec{s}}} |x_j|^{2\alpha_j} dx.
\]
Since,
\[
 \sum_{j=1}^d \int_{I_{\tau,\vec{s}}} |x_j|^{2\alpha_j} dx\le \sum_{j=1}^d \tau_j^{2\alpha_j} \prod_{k=1}^d \tau_k=\sum_{j=1}^d \tau^{\frac{2s_1\alpha_j}{s_j}} \tau^{s_1 \ell{(\vec{s})}}=d\tau^{2s_1},
\]
we conclude that
\[
\int_{I_{\tau,\vec{s}}} |w(x)|^2 dx \le d\tau^{2s_1}\max_{x \in I_{\tau,\vec{s}}} (R(|x|))^2.
\]
On the other hand, we have
\begin{align*}
	\tau^{s_1(1-1/q)\ell{(\vec{s})}}= \prod_{j=1}^d \tau^{s_1(1-1/q)/s_j}= \prod_{j=1}^d \tau_j^{1-1/q}\lesssim\|\mathcal{F}(\chi_{I_{\tau,\vec{s}}}) \|_{\ell^q(\Z^d)}.
\end{align*}
Since $u=1/w \in \mathcal{M}_2^q$ we therefore have, for any $\tau>0$,
\begin{align*}
	\tau^{s_1(1-1/q) \ell{(\vec{s})}} \lesssim \|\mathcal{F}(\chi_{I_{\tau,\vec{s}}})\|_{\ell^q(\Z^d)} \lesssim \|w\|_{L^2(I_{\tau,\vec{s}})} \lesssim \tau^{s_1} \max_{x \in I_{\tau,\vec{s}}} R(|x|).
\end{align*}
However, the condition on $q$ implies that $(1-1/q)\ell{(\vec{s})} \le 1$, and $\max_{x \in I_{\tau,\vec{s}}}R(|x|)$ tends to zero as $\tau$ tends to zero, so this bound cannot hold.
\end{proof}

\begin{remark}\label{localnonsym}
Note that due to the local nature of the proof above, Theorem \ref{thm:nonsymMult} holds if the condition $w \in H^{\vec{s}}(\T^d)$ is replaced by the condition that $w$ has a zero at $0$, and (\ref{mixed holder}) holds for all $x$ in a neighborhood of $0$.
\end{remark}

\section{Further extensions of the multiplier results}\label{sec:MatpqMult}

In this section we discuss some further extensions of our multiplier results. First, we extend these results to matrix valued multipliers, and then we remark on the extension of these results to $(p,q)$--multipliers by means of interpolation.

\subsection{Matrix valued multipliers}

Recall that for $K\in\N$ and $2\leq q\leq\infty$ a matrix valued function $U\in [L^2(\T^d)]^{K\times K}$ is a \textit{matrix valued $(2,q)$--multiplier} if the operator $T_U: [\ell^2(\Z^d)]^K\rightarrow [\ell^q(\Z^d)]^K$, defined by
\[
T_U A= \mathcal{F}_K(U\mathcal{F}^{-1}_K{A}),
\]
is bounded, and that the family of all such matrix valued multipliers is denoted by $\mathfrak{M}_2^q(K)$. Further, recall that for a Hermitian matrix valued function $U$ we write $U=V^* \Lambda V$ where the entries of $V$ and $\Lambda$ are measurable functions, $V$ is unitary, and $\Lambda$ is a diagonal matrix with diagonal entries satisfying $\lambda_1(x)\ge \lambda_2(x) \ge \cdots \ge \lambda_K(x)$ for almost every $x\in \T^d$.

We now turn to a proof of Theorem \ref{thm:KMultEig}, which asserts that we can frequently consider scalar valued multipliers instead of matrix valued multipliers.
\begin{proof}
First, assume that $U\in \mathfrak{M}_2^q(K)$. We need to show that if $\lambda(x):=\lambda_k(x)$ is some eigenvalue of $U$ then $\lambda(x)\in\MM_2^q$. Let $v(x)=(v_j(x))_{j=1}^K$ be an eigenvector of $U$ corresponding to the eigenvalue $\lambda(x)$ and such that $\sum|v_j(x)|^2=1$ for almost every $x\in\T^d$. Note that such a measurable vector valued function $v$ exists due to the decomposition $U=V^* \Lambda V$ recalled above. For a function $g\in L^2(\T^d)$ the last equality implies that $\sum|v_j|^2\lambda g =\lambda g$ almost everywhere, and so, by H\"older inequality, that there exists some $1\leq j_0\leq K$ which satisfies \[
\frac{1}{K}\|\mathfrak{F}(\lambda g)\|_{\ell^q}\leq \|\mathfrak{F}\big(|v_{j_0}|^2\lambda g \big)\|_{\ell^q}\leq \|\mathfrak{F}_K\big((\lambda g \overline{v_{j_0} })v\big)\|_{[\ell^q]^K}.
\]
Since $v$ is an eigenvector of $U$ corresponding to the eigenvalue $\lambda$ and $U$ is a matrix valued $(2,q)$--multiplier, the right hand side in the last displayed inequality satisfies
\[
\begin{aligned}
\|\mathfrak{F}_K\big((\lambda g \overline{v_{j_0}})v\big)\|_{[\ell^q]^K}=\|\mathfrak{F}_K\big(U (g \overline{v_{j_0}}v)\big)\|_{[\ell^q]^K} \leq C\|g\overline{v_{j_0}}v\|_{[L^2(\T^d)]^K}.
\end{aligned}
\]
Recalling that $\sum|v_j|^2=1$, we conclude that $\|\mathfrak{F}(\lambda g )\|_{\ell^q}\lesssim \|g\|_{L^2(\T^d)}$ and therefore that $\lambda \in\MM_2^q$.

Next we prove the reverse inclusion.  Suppose that for each $k=1,...,K$ we have, $\lambda_k \in \MM_2^q$.  In particular, this implies that if $\psi=(\psi_j)_{j=1}^K\in [L^2(\T^d)]^K$ then
\begin{equation}\label{smallbigmult}
\|\mathfrak{F}_K(\lambda_k \psi)\|_{[\ell^q]^K}\leq C\|\psi\|_{ [L^2(\T^d)]^K},\qquad k=1,...,K.
\end{equation}

For every fixed $x$, let $v^1(x),...,v^K(x)$ be an orthonormal basis of eigenvectors corresponding to the eigenvalues $\lambda_1,..,\lambda_K$.  Given $\phi\in [L^2(\T^d)]^K$ there exist measurable functions $b_1(x),...,b_2(x)$ so that $\phi(x)=\sum_{k=1}^Kb_k(x)v^k(x)$ and $\sum|\phi_k(x)|^2=\sum|b_k(x)|^2$ for almost every x. Since the vectors $v^k$ are eigenvectors, it now follows that
\[
\|\mathfrak{F}_K(U\phi)\|_{[\ell^q]^K}\leq \sum_{k=1}^K\|\mathfrak{F}_K\big(\lambda_k(x)b_k(x)v^k(x)\big)\|_{[\ell^q]^K}.
\]
The inequality in (\ref{smallbigmult}) now implies that the right hand side of the last displayed equation is less then a constant multiplying,
\[
\sum_{k=1}^K\|b_k(x)v^k(x)\|_{[L^2(\T^d)]^K}\leq K^{\frac{1}{2}}\Big(\sum_{k=1}^K\|b_k(x)v^k(x)\|^2_{[L^2(\T^d)]^K}\Big)^{\frac{1}{2}}= K^{\frac{1}{2}}\|\phi\|_{[L^2(\T^d)]^K},
\]
where the last step is due to the fact that for almost every $x\in\T^d$ we have $\sum|\phi_k(x)|^2=\sum|b_k(x)|^2$, and, for every fixed $k$, $\sum_{j}|v^k_j(x)|=1$. The result follows.
\end{proof}

Next we formulate and prove a more general version of Corollaries \ref{thm:SobolevResultSingleZero-2-2-matrix} and \ref{thm:ScalarValuedLargeZeroResult-matrix}.

\begin{corollary}[Nitzan, Northington, Powell]\label{thm:SobolevResultSingleZero-matrix-r}
 Let $1< r<\infty$, $d/r<s\le d(1/2+1/r)$ and let $W \in [L^1(\T^d)]^{K\times K}$ be a Hermitian matrix valued function whose eigenvalues are given by $\lambda_1(x)\ge \lambda_2(x) \ge \cdots \ge \lambda_K(x)$. If $\lambda_k\in W^{s,r}(\T^d)$ for every $k=1,...K$, and $\det(W)$ has a zero, then conclusions ($i$) and ($ii$) of Theorem \ref{thm:SobolevResultSingleZero-r-2} hold with $U=W^{-1}$ replacing $u=1/w$ and with $\mathfrak{M}_2^q(K)$ replacing $\mathcal{M}_2^q$. In particular, this result holds for a nonnegative matrix valued function satisfying  $W \in [{W^{s,r}(\T^d)}]^{K\times K}$.
\end{corollary}

\begin{proof}
We will prove part (i) and remark that part (ii) follows from part (i) with the same argument as part (ii) of Theorem \ref{thm:SobolevResultSingleZero-r-2}.

Since $s>d/r$, we have, by Theorem \ref{thm:FracHolderEmb}, that each $\lambda_k$ is continuous.  Thus, $\det(W)=\lambda_1\cdots \lambda_K$ has a zero if and only if for at least one $k$, $\lambda_k$ has a zero. However, since $U =W^{-1}\in \mathfrak{M}_2^q(K)$, Theorem \ref{thm:KMultEig} implies that $\lambda_k^{-1} \in \mathcal{M}_2^q$.  Therefore, it must be that $q>\frac{d}{d(1/2+1/r)-s}$ or else the function $\lambda_k$ contradicts Theorem \ref{thm:SobolevResultSingleZero-r-2}.

To prove the particular case of $W \in [W^{s,2}(\T^d)]^{K\times K}$ we may, without loss of generality, assume that $0<s<1$, since Theorem \ref{thm:SobEmb} allows us to replace $W^{s,r}$ with $W^{\tilde{s}, \tilde{r}}$ where $\tilde{s}<1$ and $s-d/r = \tilde{s}-d/\tilde{r}$.  It is straightforward to check that the assumptions on $s$ and $r$ and the restriction on $q$ are all invariant under this change. Since we assume $0<s<1$, Lemma 4.3 of \cite{HNP} implies that the eigenvalue functions $\{\lambda_k\}_{k=1}^K$ of $W$ each satisfy $\lambda_k \in W^{s,r}(\T^d)$. (In fact, this lemma is given for the space $W^{s,2}(\T^d)$, but the calculations remain unchanged when replacing the seminorm for $W^{s,2}(\T^d)$ by that of $W^{s,r}(\T^d)$.) The result now follows from the general statement of the corollary.
\end{proof}

\begin{corollary}[Nitzan, Northington, Powell]\label{thm:LargeZeroResult-matrix}
Let $2\le r<\infty$, $0< \sigma <d$, $(d-\sigma)/r < s \le 1$, and let $W \in [L^1(\T^d)]^{K\times K}$ be a Hermitian matrix valued function whose eigenvalues are given by $\lambda_1(x)\ge \lambda_2(x) \ge \cdots \ge \lambda_K(x)$. If $\lambda_k\in W^{s,r}(\T^d)$ for every $k=1,...K$, and $\det(W)$ has a zero, then conclusions ($i$) and ($ii$) of Theorem \ref{thm:ScalarValuedLargeZeroGeneralResult} hold with $U=W^{-1}$ replacing $u=1/w$ and with $\mathfrak{M}_2^q(K)$ replacing $\mathcal{M}_2^q$. In particular, this result holds for a nonnegative matrix valued function satisfying  $W \in [W^{s,r}(\T^d)]^{K\times K}$.
 \end{corollary}

\begin{proof}
It suffices to prove that if $\mathcal{H}^\sigma(\Sigma(\det(W)))>0$, then for the smallest eigenvalue $\lambda_K$ of $W$ we have $\mathcal{H}^\sigma(\Sigma(\lambda_K)))>0$, for then the result follows from Theorems \ref{thm:KMultEig} and \ref{thm:ScalarValuedLargeZeroGeneralResult}. For $\tau>0$ and $x_0\in\T$ consider $I_{\tau}:=I_{\tau}(x_0)$. We have,
\[ \frac{1}{|I_\tau|} \int_{I_\tau} |\lambda_K(x)| dx \le \left( \frac{1}{|I_\tau|} \int_{I_\tau} |\lambda_K(x)|^K dx\right)^{1/K}\le \left( \frac{1}{|I_\tau|} \int_{I_\tau} |\det(W)(x)| dx\right)^{1/K},\]
and therefore, $\Sigma(\det{W})\subset\Sigma(\lambda_K)$.

When $0<s<1$, the particular case of $W \in [W^{s,r}(\T^d)]^{K\times K}$ follows by directly applying Lemma 4.3 of \cite{HNP}, as was described in the proof of the previous corollary. For $s=1$, the same lemma gives the bound
\[ |\lambda_k(x)-\lambda_k(y)|\le \sqrt{\sum_{i,j} |W_{ij}(x)-W_{i,j}(y)|^2}. \]
Theorem 5.8.3 of \cite{Evans}, which equates $L^p$--norms of difference quotients to $L^p$--norms of distributional derivatives, combined with the equation above, implies that $\lambda_k \in W^{1,r}(\T^d)$.
\end{proof}

\subsection{A remark regarding \texorpdfstring{$(p,q)$--multipliers}{(p,q)--multipliers}} \label{subsec:PQMult}

Given $1\leq p\leq q\leq \infty$, we say that a distribution $u$ is a $(p,q)$--\textit{multiplier} if the operator $T_u$, defined by
\begin{equation}\label{eqn:defPQMult}
T_u a= \mathcal{F}(u\mathcal{F}^{-1}{a}),
\end{equation}
is a bounded operator from  $\ell^p(\Z^d)$ to $\ell^q(\Z^d)$. The family of all such multipliers is denoted by $\mathcal{M}_p^q$. Endowed with the operator norm $\MM_p^q$ is a Banach space.  Such spaces were studied by A. Devinatz and I. I. Hirschman Jr. \cite{DH}, and their analog over $\R^d$ was studied by L. H\"ormander \cite{H}. The space  $\mathfrak{M}_p^q(K)$ of matrix valued $(p,q)$--multipliers can be defined similarly.

  Part (i) of the following proposition is a special case of Theorem 1.3 in \cite{H}. In part (ii) of the proposition we  extend this result to the matrix valued multiplier setting. As the proof in both cases is similar, we omit it.
\begin{proposition}[Part (i) appears in \cite{H}] \label{prop:Reduction2qMult}
Suppose that $p$ and $q$ satisfy either $1 \le p \le q \le 2$ or $2 \le p \le q \le \infty$ and denote $\tilde{q}=(1/2-1/p+1/q)^{-1}$. Then,
\begin{itemize}
\item[i.] If $u \in \mathcal{M}_p^q$, then $u \in \mathcal{M}_{2}^{\tilde{q}}$.
\item[ii.] If $U \in \mathfrak{M}_p^q(K)$, then $U \in \mathfrak{M}_{2}^{\tilde{q}}(K)$.
\end{itemize}
\end{proposition}
It follows from this lemma that if $1 \le p \le q \le 2$ or $2\le p \le q \le \infty$ then $\mathcal{M}_p^q\subset L^2(\T^d)$ and, in particular, contains only functions. We therefore focus our attention on these cases. Combining Proposition \ref{prop:Reduction2qMult} with theorems \ref{thm:SobolevResultSingleZero-r-2} and \ref{thm:ScalarValuedLargeZeroGeneralResult}, as well as with corollaries \ref{thm:SobolevResultSingleZero-2-2-matrix} and \ref{thm:ScalarValuedLargeZeroResult-matrix}, one can obtain $(p,q)$--multiplier versions of these results. To illustrate, we give in the corollary below a $(p,q)$--multiplier version of Theorem \ref{thm:SobolevResultSingleZero-r-2}.

\begin{corollary}[Nitzan, Northington, Powell] \label{cor_Mpq}
 Let $2\le r<\infty$, $d/r< s \le d(1/p+1/r)$. Suppose $w \in W^{s,r}(\T^d)$ and $w$ has a zero.
\begin{itemize}
	\item[i.] If $ s<\frac{d}{r}+1$, then $u=\frac{1}{w} \notin \mathcal{M}_p^q$ for any $q$ satisfying $2 \le q \le \frac{d}{d(1/p+1/r)-s}$.  Conversely, for any $0<s\le d$ and  $q>\frac{d}{d(1/p+1/r)-s}$, there exists $w \in W^{s,r}(\T^d)$ such that $w$ has a zero and $u=\frac{1}{w}\in \mathcal{M}_p^q$.
	\item[ii.] If $s=\frac{d}{r}+1$, then $u=\frac{1}{w} \notin \mathcal{M}_p^q$ for any $q$ satisfying $2 \le q<\frac{pd}{d-p}$.
	\end{itemize}
\end{corollary}

\section{Applications to Time--Frequency Analysis: Gabor Systems}\label{sec:TFApp}

In this section, we first relate the Fourier multiplier property to basis properties of exponential systems in weighted spaces.  Then, we use the Zak transform to connect between Gabor systems and such exponential systems, and prove Balian--Low type theorems in this setting.

\subsection{Exponential systems in weighted spaces}\label{sec:BasisProps}
Let $w\in L^1(\T^d)$ be a nonnegative function. The corresponding weighted space $L^2_w(\T^d)$, is the Hilbert space which consists of all functions $f$ satisfying $\|f\|^2_{L^2_w}:=\int_{\T^d}|f|^2wdx<\infty$.
Consider the set of exponentials with integer frequencies
\[E= \{e_{n}\}_{n \in \Z^d}:=\{ e^{2\pi i \langle n,x \rangle}\}_{n \in \Z^d},\]
and note that, since $w\in L^1(\T^d)$, it both belongs to the space $L^2_w(\T^d)$ and is complete there.
In fact, many other basis properties of $E$ can be characterized in terms of the weight function $w$.  The following proposition lists a few of these.

\begin{proposition}[e.g. Theorem 10.10 in \cite{HPrimer}]\label{char}
Let $L^2_w(\T^d)$ and $E$ be as above. Then,
\begin{itemize}
	\item[i.] $E$ is a Riesz basis if and only if $w, 1/w \in L^\infty(\T^d)$.
	\item[ii.] $E$ is exact (complete and minimal) if and only if $1/w \in L^1(\T^d)$.
\end{itemize}
\end{proposition}

Recall that for $q\geq 2$, a system $\{f_n\}$ in a Hilbert space $H$ is a ($C_q$)--system
if there exists $C>0$ such that every $f\in H$ can be approximated arbitrarily well by a finite linear combination $\sum a_nf_n$ with $\|a_n\|_{\ell^q}\leq C\|f\|_H$. Further, recall that
$\{f_n\}$ is an exact ($C_q$)--system if and only if it is complete and there exists $D>0$ such that
\begin{equation}\label{Cq-cond-1}
D \left(\sum|a_n|^q\right)^{\frac{1}{q}}\leq \left\|\sum a_nf_n \right\|_H,
\end{equation}
for any finite sequence $\{a_n\}$. We note that the system $E$ is exact in $L^2_w(\T^d)$ if and only if it is an exact ($C_{\infty}$)--system, while it is a Riesz basis in the space if and only if it is a (Bessel) exact ($C_2$)--system there. In the following proposition, we extend the characterizations from Proposition \ref{char} to exact $(C_q)$--systems for all $2\leq q\leq\infty$.

\begin{proposition}\label{prop:MinChar}
Let $L_w^2(\T^d)$ and $E$ be as above.
Then,  $E$ is an exact $(C_q)$--system for $L^2_w(\T^d)$ if and only if  $w^{-1/2} \in \mathcal{M}_{2}^{q}$.
\end{proposition}
\begin{proof}
We first note, as we have seen in previous proofs, that the condition $w^{-1/2} \in \mathcal{M}_{2}^{q}$ can be reformulated as follows: $w\neq 0$ almost everywhere and,
\begin{equation}\label{Cq-cond-2}
D \|\widehat{f}\|_{\ell^q(\Z^d)}\leq \|f\|_{L_w^2(\T^d)},\qquad  \forall f\in L_w^2(\T^d)
\end{equation}
where $D$ is some positive constant. Now, assume first that $w^{-1/2} \in \mathcal{M}_{2}^{q}$. Then, for sums of the form $f=\sum a_ne_n$, the conditions in (\ref{Cq-cond-2}) is exactly the same as the condition in (\ref{Cq-cond-1}). Since the system $E$ is complete in the space, it follows that it is an exact $(C_q)$--system there. Conversely, assume that $E$ is an exact $(C_q)$--system  in the space then, in particular, Proposition \ref{char} implies that $w\neq 0$ almost everywhere. Moreover, for sums of the form $f=\sum a_ne_n$ the condition in (\ref{Cq-cond-2}) is implied by the condition in  (\ref{Cq-cond-1}).
By a usual limiting procedure (\ref{Cq-cond-2}) holds for all $f\in L_w^2(\T^d)$, which implies that $w^{-1/2} \in \mathcal{M}_{2}^{q}$. This completes the proof.
\end{proof}

\subsection{Gabor systems}\label{subsec:Gabor}

In this subsection we prove the following nonsymmetric generalization of Theorem \ref{thm:CqGabor}. Its novelty is in obtaining end point results for conditions previously found in \cite{NO1} (see  Theorem 2 there).

\begin{theorem}[Nitzan, Northington, Powell]\label{thm:nonsymmetricBLT}
Let $2<q<\infty$, $q'=q/(q-1)$, and $t\geq r>{4(q-1)}/{(q+2)}$ be such that
$1/{r}+ 1/{t} \le {q'}/{2}.$
If $G(g)$ is an exact $(C_q)$--system for $L^2(\R)$, then either
\begin{align}\label{eqn:CqGaborBLTIntegralsNonsymmetric}
	\int_{\R} |x|^{r} |g(x)|^2 dx = \infty \text{   or   } \int_{\R} |\xi|^{t} |\widehat{g}(\xi)|^2 d\xi = \infty.
\end{align}
The theorem also holds with $r$ and $t$ interchanged in \eqref{eqn:CqGaborBLTIntegralsNonsymmetric}.
\end{theorem}

In \cite{Gautam}, S. Z. Gautam showed that the above result holds when ``exact $(C_q)$--system" is replaced by ``Riesz basis" where $q$ is replaced by $2$ in the parameters restrictions above. Similarly, in \cite{HPExact}, C. Heil and A. M. Powell proved that the above result holds for exact systems where the range of parameters corresponds to $q=\infty$ in the inequalities above.

It follows from Theorem 2 in \cite{NO1} that \eqref{eqn:CqGaborBLTIntegralsNonsymmetric} holds when $(r,t)$ falls into a certain region in the plane. Theorem \ref{thm:nonsymmetricBLT} extends this result to include part of the boundary of this region (which corresponds to the interval GF in Figure 1 in \cite{NO1}).  Theorem 2 in \cite{NO1} shows  also that this result is  sharp in the sense that it doesn't hold for $1/r+1/t >q'/2$.  Note that when $r=t$, we have $r=4/q'$, which is exactly the result in Theorem \ref{thm:CqGabor}.

Our proof of Theorem \ref{thm:nonsymmetricBLT} relies on properties of the Zak transform. For a continuous function $g$ with sufficient localization The Zak transform is defined by
\[ Zg(x,y)= \sum_{k \in \Z} g(x-k) e^{2\pi i k y},\]
and extended to a unitary operator from $L^2(\R)$ to $L^2([0,1]^2)$ in the standard way.
Note that $Zg$ is quasi--periodic in the sense that
\begin{equation}\label{eq:ZacProperties} Zg(x,y+1)= Zg(x,y), \indent Zg(x+1,y)=e^{2\pi i y} Zg(x,y).\end{equation}
In particular, $|Zg|$ is a periodic function. This quasi--periodicity property implies that if $Zg$ is continuous, then it must have a zero (see e.g.  \cite{G}).

It is readily checked that if $g\in L^2(\R)$ and $G(g)$ is complete in $L^2(\R)$, then the mapping $U_g:L^2(\R)\rightarrow L^2_{|Zg|^2}(\T^2)$ defined by
\[ U_g h = \frac{Zh}{Zg}\]
is an isometric isomorphism, and the image of $G(g)$ under $U_g$ is the system $E$  (see e.g.\cite{NO1}).  Combining this with Proposition \ref{prop:MinChar} yields the following.
\begin{proposition}\label{prop:CqGaborFM}
Let $q\geq 2$ and $g \in L^2(\R)$.  Then $G(g)$  is an exact $(C_q)$--system in $L^2(\R)$ if and only if $\frac{1}{|Zg|} \in \mathcal{M}_2^q$.
\end{proposition}

The next proposition relates time--frequency properties of a function $g\in L^2(\R)$ to smoothness properties of its Zak transform. This proposition is proved in  \cite{NO1} (Lemma 5), though stated there in a different terminology. In particular, in  \cite{NO1}, Sobolev conditions over $\R^2$ were considered, the corresponding conditions over $\T^2$ may be obtained by means of the Parseval equality.

\begin{proposition}[\cite{NO1}]\label{local}
Let $g\in L^2(\R)$ be such that both integrals in (\ref{eqn:CqGaborBLTIntegralsNonsymmetric}) are finite for $g$. Denote $s_1=r/2$ and $s_2=t/2$. Then, for every $x\in \R^2$ there exist $\tau>0$ and $u\in  H^{(s_1,s_2)}(\T^d)$ (considered as a periodic function over $\R^2$) so that $Zg=u$ over $B_{\tau}(x)$.
\end{proposition}

\begin{proof}[Proof of Theorem \ref{thm:nonsymmetricBLT}]

{First, we note that it is enough to prove the theorem for $t\geq r>{4(q-1)}/{(q+2)}$ satisfying
\begin{equation}\label{for ell}
\frac{1}{r}+ \frac{1}{t} = \frac{q'}{2}.
\end{equation}
As the lines $u+v=q'/2$ and $u+3v=1$ intersect at $u=(q+2)/4(q-1)$, we have that for all $r$ and $t$ at the prescribed range:
\begin{equation}\label{for alpha}
\frac{1}{r}+\frac{3}{t}>1.
\end{equation}
Now, suppose for a contradiction that both integrals in \eqref{eqn:CqGaborBLTIntegralsNonsymmetric} are finite, and denote $r=2s_1$ and $t=2s_2$. It follows from  Proposition \ref{local} and Theorem \ref{thm:nonsymSobolevEmb} that $Zg$ is continuous and therefore that $|Zg|$  has a zero in $[-1/2, 1/2)^2$. Without loss of generality, we may assume that $Zg(0)=0$. Note that condition (\ref{for ell}) implies that $\ell:=1/s_1+1/s_2=q'$, while condition (\ref{for alpha}) implies that $0< \alpha_j=s_j( 1-\ell/{2})<1$ for $j=1,2$. It therefore follows from Proposition \ref{local} and Theorem \ref{thm:nonsymSobolevEmb} that in a neighbourhood of $0$,
\begin{equation}\label{mixed holder-2}
	|Zg(x,y)|\le R(|(x,y)|) \big(|x|^{\alpha_1}+|y|^{\alpha_2}\Big),
\end{equation}
where $R(\tau)$ tends to zero as $\tau$ tends to zero.}

As $G(g)$ is an exact $(C_q)$--system, Proposition \ref{prop:CqGaborFM} implies that $\frac{1}{|Zg|} \in \mathcal{M}_2^q(\T^d)$. This, combined with the conditions in (\ref{mixed holder-2}), contradicts the conclusion of Remark \ref{localnonsym}.
\end{proof}

\section{Applications to Time--Frequency Analysis: Shift--Invariant spaces}\label{subsec:SIS}

\subsection{Exponentials in multi--variable weighted spaces}\label{sec:BasisProps-2}

Let $W$ be a $\Z^d$--periodic, $K\times K$ matrix valued function, which is positive--definite for almost every $x\in \T^d$ and for which
\begin{equation}\label{eqn:traceInt}
 \text{tr}(W) \in L^1(\T^d).
\end{equation}
We refer to such a $W$ as a matrix valued weight function. We denote by $L^2_W(\T^d)$ the space of all vector--valued functions $\psi=(\psi_1, \psi_2,..., \psi_K)^T$, with $\psi_k$ defined on $\T^d$, which satisfy
\[\|\psi\|^2_{L^2_W(\T^d)} = \int_{\T^d}\langle W\psi, \psi\rangle dx<\infty.\]
 With the implied inner product, $L^2_W(\T^d)$ is a Hilbert space.

Let $e_k$ denote the $k^{th}$ canonical (column) basis vector in $\mathbb{C}^d$, and define
\[E(K)=\{e_k e^{2\pi i \langle n, x \rangle}\}_{k \in \{1,...,K\}, n \in \Z^d}= \{e_{k,n}\}_{k \in \{1,...,K\}, n \in \Z^d}.\]
It is readily checked that condition \eqref{eqn:traceInt} is both necessary for $E(K)$ to be a subset of $L^2_W(\T^d)$ and sufficient for $E(K)$ to be complete in $L^2_W(\T^d)$. Such spaces were previously studied in \cite{HSWW}.

Part (i) of Proposition \ref{char} is extended in \cite{RS} to the multi--variable case (see Theorem 2.3.6). It is proved there that $E(K)$ forms a Riesz basis in $L^2_W(\T^d)$ if and only if each eigenvalue function, $\zeta_k$, of $W$ satisfies $0<A\le \zeta_k(x)\le B<\infty$ for almost every $x\in \T^d$. Keeping Theorem \ref{thm:KMultEig} in mind, we conclude that  $E(K)$ forms a Riesz basis in $L^2_W(\T^d)$ if and only if $W(x)$ is uniformly bounded in norm, invertible for almost every $x$, and $W^{-1/2} \in \MM_{2}^{2}(K)$. Our next goal is to obtain a similar extension of Proposition \ref{prop:MinChar}. We will use the following lemma.

\begin{lemma}\label{stupidchar}
Let $W$ be a $K \times K$ matrix valued weight function. If $E(K)$ is exact in $L^2_W(\T^d)$ then $W^{-1}(x)$ is defined for almost every $x$, and is a measurable matrix valued function.
\end{lemma}
\begin{proof}
Assume that $E(K)$ is exact in $L^2_W(\T^d)$. For every $k=1,...,K$ let $h_k\in L^2_W(\T^d)$ be the vector valued function satisfying
\[
\langle h_k, e_{j,\ell}\rangle_{L^2_W(\T^d)} =\left\{\begin{array}{ccc}
1$\qquad$(j,\ell)=(k,0)\\
0$\qquad\qquad$\textrm{otherwise}\end{array}\right.
\]
Since $E(K)$ is an orthonormal basis in $[L^2(\T^d)]^K$, it follows that $W(x)h_k(x)=e_{k,0}(x)$ for every $k=1,..,K$ and almost every $x$. We conclude that $WH=I$ almost everywhere, where $I$ is the constant identity matrix and $H$ is the matrix whose $k$'th column is $h_k$.
\end{proof}

Applying Lemma \ref{stupidchar}, the following proposition can be proved in exactly the same way as Proposition \ref{prop:MinChar}. We omit the proof.
\begin{proposition}\label{prop:MinChar-multi}
Let  $W$ be a $K \times K$ weight function.
 Then, $E$ is an exact $(C_q)$--system for $L^2_W(\T^d)$ if and only if $W$ is invertible almost everywhere and $W^{-1/2} \in \MM_{2}^{q}(K)$
\end{proposition}

\subsection{The Gramian}

For a vector valued function $H=(h_1, ..., h_K)^T\in [L^2(\R^d)]^K$, we denote by $P(H)$ the following $\Z^d$--periodic, positive--semidefinite $K\times K$ matrix valued function,
\begin{equation*}
	P(H)(x)= \sum_{k \in \Z^d} H(x+k) H(x+k)^*,
\end{equation*}
  where $A^*$ is the adjoint matrix to $A$.
  Note that in the case of a single function $P(h)(x)=\sum_{k\in\Z^d} |h(x+k)|^2$.

 For $F=\{f_1,..,f_K\}\in L^2(\R)$ recall the notations,
 $$V(F)=\overline{\textrm{span} \thinspace \mathcal{T}(F)},\qquad\mathcal{T}(F)=\{f_k(x-n):n\in\Z^d, k=1,..,K\}.$$
  With a slight abuse of notation, we denote $\widehat{F}(x)=(\widehat{f}_1(x), \widehat{f}_2(x),... ,\widehat{f}_K(x))^T$. Similar to the Zak tranform for Gabor systems, many properties of $\mathcal{T}(F)$ as a system in the space $V(F)$ may be characterized in terms of $P(\widehat{F})$, which we refer to as the Gramian of $F$. Indeed, define the mapping $I_F: L^2_{P(\widehat{F})}(\T^d)\rightarrow V(F)$ by,
	\[ \widehat{I_F M}= M \widehat{F}= m_1 \widehat{f}_1+ \cdots m_K \widehat{f}_K.\]  Then, (e.g. \cite{HSWW}) $I_F$ is an isometric isomorphism between $L^2_{P(\widehat{F})}(\T^d)$ and $V(F)$.  Note that $I_F e^{2\pi i \langle n, \xi\rangle } e_k = f_k(x-n)$.  Therefore, the set $\mathcal{T}(F)$ in $V(F)$ corresponds to the set of exponentials $E(K)=\{ e_k e^{2\pi i \langle n , \xi\rangle }\}_{k \in \{1,..., K\}, n \in \Z^d}$ in the space $L^2_{P(\widehat{F})}(\T^d)$.  Thus, we have the following from Proposition \ref{prop:MinChar}.

\begin{proposition}\label{prop:CqSISFM}
Fix $q\ge 2$ and let $F \subset L^2(\R^d)$.  Then, $\mathcal{T}(F)$ forms an exact $(C_q)$--system for $V(F)$ if and only if $P^{-1/2}(\widehat{F}) \in \mathfrak{M}_{2}^{q}(K)$.
\end{proposition}

  \subsection{Smoothness properties of the Gramian}

  The following proposition relates smoothness properties of a vector--valued function, $H$, to smoothness properties of the corresponding matrix, $P(H)$.

  \begin{proposition}[Nitzan, Northington, Powell]\label{prop:sqrt-per-emb}
Fix $0<s\le1$.
Suppose $H \in [W^{s,2}(\R^d)]^K$ and $\zeta_1(x)\ge\cdots\ge \zeta_K(x)\ge 0$ are the eigenvalues of $P(H)(x)$, then $\sqrt{\zeta_k} \in W^{s,2}(\T^d)$ for every $k = 1,...,K$.
\end{proposition}

\begin{proof}
For $H=(h_1,...,h_K) $,  let $A(x)$ be the operator mapping $\C^K$ into sequences on $\Z^d$, defined by
\[ (A(x) c)_\ell = c_1\overline{h_1(x-\ell)} + c_2 \overline{h_2(x-\ell)} +... + c_K\overline{h_K(x-\ell)}.\]
 Note that $A(x)$ may be viewed as an $\infty \times K$ matrix, and that
\[ P(x):=P(H) (x)= A(x)^* A(x),\]
which implies the equality $ \langle P(x)c, c\rangle=\|A(x)c\|_2^2$ for every $c\in \C^K$. The min--max theorem (see e.g., Corollary III.1.2 in \cite{BH}), shows that the eigenvalues $\zeta_k$ of $P(H)$ satisfy
\begin{align*}
\sqrt{\zeta_k(x)}=\max\{\min\{  \|A(x)c\|_2 : c \in M, |c|=1\}: \dim(M)=k\}.
\end{align*}
From this equation we immediately obtain the bound
\[\zeta_k(x) \le \sum_{\ell \in \Z^d} \sum_{j=1}^K |h_j(x-\ell)|^2\qquad k=1,...,K.\]
In particular, $\sqrt{\zeta_k} \in L^2(\T^d)$ for every $k=1,...,K$.

We will now obtain a similar estimate for $|\sqrt{\zeta_k(x)}-\sqrt{\zeta_k(y)}|$.  Without loss of generality, we may assume $\zeta_k(x)\ge \zeta_k(y)$.  Choose a subspace $M_0$ which realizes the maximum for $\zeta_k(x)$.  Then, we have
\begin{align*}
	 \left|\sqrt{\zeta_k(x)}-\sqrt{\zeta_k(y)}\right|&\le  \min\{  \|A(x)c\|_2 : c \in M_0, |c|=1\}-\min\{ \|A(y)b\|_2 : b \in M_0, |b|=1\}.
\end{align*}
Next, choose $b_0 \in M_0$ with $\|b_0\|=1$ such that the minimum in the right term is achieved at $b_0$.  It follows that,
\begin{align}
\left|\sqrt{\zeta_k(x)}-\sqrt{\zeta_k(y)}\right|& \le  \|A(x)b_0\|_2 -\|A(y)b_0\|_2\le \| (A(x)-A(y)) b_0\|_2 \nonumber\\
	& \le \left(\sum_{j=1}^K \sum_{\ell \in \Z^d} |h_j(x-\ell)-h_j(y-\ell)|^2\right)^{1/2}.\label{eqn:eigbound}
\end{align}

\underline{\textit{Case 1: $s<1$}} Using Equation \eqref{eqn:eigbound}, we find
\begin{align*}
 \|\sqrt{\zeta_k}\|^2_{\dot{W}^{s,2}(\T^d)}&\le C \int_{[-\frac{1}{2},\frac{1}{2}]^d} \int_{[-\frac{1}{2}, \frac{1}{2}]^d} \frac{\sum_{j=1}^K \sum_{\ell \in \Z^d} |h_j(x+y-\ell)-h_j(x-\ell)|^2}{|y|^{d+2s}} dy dx\\
		&= C \sum_{j=1}^K \int_{\R^d} \int_{[-\frac{1}{2}, \frac{1}{2}]^d} \frac{ |h_j(x+y)-h_j(x)|^2}{|y|^{d+2s}} dy dx \le C \sum_{j=1}^K \|h_j\|_{\dot{W}^{s,2}(\R^d)}^2 ,
\end{align*}
and the right hand side is finite by the assumptions on $H$.

\underline{\textit{Case 2: $s=1$}} Here we use the equivalence between the spaces $W^{s,2}$ and the spaces $H^{s}$. For notational simplicity, let $g=\sqrt{\zeta_k}$. Equation \eqref{eqn:eigbound} implies that for any $i\in\{1,...,K\}$
\begin{align*}
		\int_{\T^d}\left| g(x+te_i) - g(x)\right|^2 dx &\le \int_{\T^d}\sum_{j=1}^K \sum_{\ell \in \Z^d} \left|h_j(x+te_i-\ell)-h_j(x-\ell)\right|^2dx\\
		&=  \sum_{j=1}^K \int_{\R^d} \left|h_j(x+te_i)-h_j(x)\right|^2dx.
\end{align*}
Parseval's equality for $L^2(\R^d)$ and $L^2(\T^d)$ allows us to reformulate this inequality as
\begin{align*}
		\sum_{n \in \Z^d} |\widehat{g}(n)|^2 |e^{2\pi i n_i t}-1|^2
		& \le  \sum_{j=1}^K \int_{\R^d} |\widehat{h}_j(\xi)|^2 |e^{2\pi i \xi_i t}-1|^2 d\xi.
\end{align*}
Standard bounds on $|e^{2\pi i \xi_i t}-1|$ now imply that,
\begin{align*}
	\sum_{|n|\le \frac{1}{4|t|}} |\widehat{g}(n)|^2 |n_i|^{2} \lesssim  \sum_{n \in \Z^d} |\widehat{g}(n)|^2 \left|\frac{e^{2\pi i n_i t}-1}{t}\right|^2  \lesssim  \sum_{j=1}^K \int_{\R^d} |\widehat{h}_j(\xi)|^2 |\xi_i |^2 d\xi
\end{align*}
and the right hand side is uniformly bounded in $t$.  Thus, taking the limit as $t \rightarrow 0$, and summing over $i \in \{1,...,K\}$ we find that
\[ \|g\|_{\dot{H}^{1}(\T^d)}^2 \lesssim \sum_{j=1}^K \|h_j\|_{\dot{H}^{1}(\R^d)}^2<\infty.\]
\end{proof}

\subsection{Extra invariance}

In this subsection we show that if a shift--invariant space $V(F)$ has non--trivial extra invariance then the determinant of $P(\widehat{F})$ has a 'large' zero set.

We will require some preliminary definitions and results. Throughout this section the term \textit{full--rank lattice}, or in short \textit{lattice}, refers to a set $\Gamma=B\Z^d$, where $B$ is a $d\times d$ real invertible matrix. For such a lattice $\Gamma$, the dual lattice $\Gamma^*$ is defined by
\[\Gamma^*=\{y\in\R^d:\forall x\in\Gamma,\:\: \langle x, y\rangle  \in \Z\},\] and satisfies $\Gamma^*=(B^T)^{-1}\Z^d$.

Let $F=\{f_1,...,f_K\}\subset\R^d$ and consider the shift--invariant space $V(F)$ and the Gramian $P(\widehat{F})$. For a lattice $\Z^d\subsetneq \Gamma\subset \R^d$, let $R\subset\Z^d$ be a set of representatives of the quotient $\Z^d/\Gamma^*$. By rearranging terms in the sum we can rewrite the Gramian as
\[
	P(\widehat{F})(x)= \sum_{k \in R}\sum_{\gamma \in \Gamma^*} \widehat{F}(x+\gamma+k)\widehat{F}(x+\gamma+k)^*=\sum_{k \in R}  P_{\Gamma^*}(\widehat{F})(x+k),
\]
where we define $P_{\Gamma^*}(\widehat{F})(x)= \sum_{\gamma \in \Gamma^*} \widehat{F}(x+\gamma)\widehat{F}(x+\gamma)^*$.

Recall that the space $V(F)$ has extra invariance if there exists $\gamma\in\R^d\setminus\Z^d$ such that for every $h\in V(F)$ we have also $h(x-\gamma)\in V(F)$. The space $V(F)$ has extra invariance if and only if there exists a lattice $\Z^d\subsetneq \Gamma\subset \R^d$ such that $V(F)$ is invariant under translates of all elements in $\Gamma$. Indeed, if $V(F)$ is invariant to translates by $\gamma$ then it is invariant to translates by all elements in the closed additive group generated by $\Z^d$ and $\gamma$, see Proposition 2.1 of \cite{ACP}.
Now, if $\gamma$ is a rational point one can choose $\Gamma$ to be this group, while if $\gamma$ has an irrational component, then there exists a rational point $\tilde{\gamma}$ which belongs to this group but not to $\Z^d$, and $\Gamma$ may be chosen to be the additive group generated by $\tilde{\gamma}$ and $\Z^d$.

 For a lattice $\Gamma \supsetneq \Z^d$ the $\Gamma$--invariance of $V(F)$  is characterized in terms of $P(\widehat{F})$ by A. Aldroubi, C. Cabrelli, C. Heil, K. Kornelson, and U. Molter in dimension one, \cite{ACHKM}, and by M. Anastasio, C. Cabrelli, and V. Paternostro in higher dimensions, \cite{ACP}.
The space $V(F)$ is $\Gamma$--invariant if and only if
\begin{equation}\label{eqn:RankFormula}
{\rm rank}\left[ P(\widehat{F})(x) \right]= \sum_{k \in R} {\rm rank} \left[ P_{\Gamma^*}(\widehat{F})(x+k) \right], \text{ a.e. }  x \in \R^d.
\end{equation}
In particular, this identity combined with the following observation of R. Tessera and H. Wang, which is a simple application of the min--max theorem, implies an estimate on the eigenvalues of the corresponding matrices.
\begin{lemma}[Lemma 3.1 in \cite{TW}]\label{ev-estimate}
Let $A_1,...,A_K$ and $B$ be nonnegative matrices and denote by $\eta_k$ and $\mu$ the smallest positive eigenvalue of $A_k$ and $B$, $1\leq k\leq K$, respectively. If $B=\sum A_k$ and $\rm{rank}(B)=\sum \rm{rank}(A_k)$ then $\mu\leq \min_{1\leq k\leq K}\eta_k$.
\end{lemma}

Next, denote by $J$ the cardinality of the smallest set $H \subset L^2(\R^d)$ such that $V(H)=V(F)$.  Proposition 4.1 in \cite{TW} implies that
\begin{equation}\label{eqn:MinGen}
 J = \text{ess} \sup_{x \in \R^d}\left( \text{rank}\left[ P(\widehat{F}) (x)\right]\right).
 \end{equation}
Recall that the extra invariance by $\gamma$ of a space $V(F)$ is non--trivial if $J\gamma\notin\Z^d$.   It follows from the discussion above that $V(F)$ has non--trivial extra invariance if and only if there exists a lattice $\Z\subsetneq \Gamma\subset \R^d$ such that $V(F)$ is invariant under translates of all elements in $\Gamma$ and the size of the quotient group $\Gamma/\Z^d$, denoted by $[\Gamma:\Z^d]$, does not divide $J$. Indeed, to see this note that the smallest integer $m$ for which $m\gamma\in\Z^d$ is equal to $[\Gamma:\Z^d]$ where $\Gamma$ is the closed group generated by $\gamma$ and $\Z^d$.

We will require also the following simple geometric observation. For $x \in \R^{d-1}$ recall the notation $L_i(x)$ for the line $L_i(x)=\{(x_1,...,x_{i-1},t,x_{i},...,x_{d-1})\}$.  (See Section \ref{subsec:SobLines})

\begin{lemma}\label{linelemma}
 Let $d\geq 2$ and $S\subset \R^d$ be a measurable set which satisfies $|S|,|S^c|>0$. Then, for some $i_0 \in\{1,...,d\}$, there exists a set of positive $(d-1)$-dimensional Lebesgue measure $A \subset \R^{d-1}$ so that for every $x\in A$ we have
\begin{equation}\label{lineintersect}
S\cap L_{i_0}(x)\neq\emptyset\qquad\textrm{and}\qquad S\cap L_{i_0}(x)\neq L_{i_0}(x).
\end{equation}
\end{lemma}

\begin{proof}
	Throughout this proof we refer to the Lebesgue measure in $\R^n$ as the $n$--dim measure.
 Denote $S_1=\{x\in\R^{d-1}: S\cap L_1(x)=L_1(x)\}$. Note that if (\ref{lineintersect}) does not hold for $i=1$, then $S=\R\times S_1$ $d$--dim almost everywhere, and that in this case the assumption on $S$ implies that both $S_1$ and $\R^{(d-1)}\setminus S_1$ have positive $(d-1)$--dim  measure.

 We first consider the case $d=2$. If (\ref{lineintersect}) does not hold for $i=1$ then  for almost every $x\in\R$ the restriction of $S$ to $L_2(x)$ is equal to $S_1$ $1$--dim almost everywhere, and therefore (\ref{lineintersect})  holds for $i=2$. We proceed by induction and assume the Lemma holds for $d-1$. If (\ref{lineintersect}) does not hold for $i=1$ then, as observed above, $S_1$ satisfies the conditions of the Lemma in $\R^{d-1}$. It follows that there exists $i_0\in\{2,...,d\}$ and a set $A'\subset \R^{d-2}$ of positive $(d-2)$--dim measure, so that in $\R^{d-1}$ condition (\ref{lineintersect}) holds for $S_1$ and $A'$. Since  $S=\R\times S_1$ $d$--dim almost everywhere, we conclude that in $\R^d$ condition (\ref{lineintersect}) holds for $S$ and almost every $x\in A:=\R\times A'$. The fact that $A'$ has positive $(d-2)$--dim measure implies that $A$ has positive $(d-1)$--dim measure, and the conclusion of the lemma follows.
\end{proof}

We are now ready to prove the following proposition which generalizes previous one--dimensional results of \cite{ASW, TW} to higher dimensions.

\begin{proposition}\label{prop:LargeZeroSet}
Fix  $1/2<s\le1$. Let $F=\{f_1,...,f_K\}\subset\R^d$ and let $J$ be the cardinality of the smallest set $G \subset L^2(\R^d)$ such that $V(G)=V(F)$. If $V(F)$ has a non--trivial extra invariance and $\widehat{f}_k\in W^{s,2}(\R^d)$ for $k=1,...,K$, then
$\mathcal{H}^{d-1}(\Sigma(\zeta_{J}))>0$ for the eigenvalue $\zeta_{J}$ of $P(\widehat{F})$.
\end{proposition}

\begin{proof}
Let $\Z^d\subsetneqq \Gamma$ be a lattice such that $V(F)$ is invariant under shifts by the elements of $\Gamma$, and let $\Gamma^*$ be its dual. Such a lattice exists due to the discussion above.
To simplify notations we denote $H=\widehat{F}$, $P=P(H)$ and $P_{\Gamma^*}=P_{\Gamma^*}(H)$.
Let $\zeta_1(x)\ge \cdots\ge  \zeta_K(x)\ge 0$ denote the eigenvalues of $P(x)$ and $\gamma_1(x) \ge \cdots \ge \gamma_K(x) \ge 0$ denote the eigenvalues of $P_{\Gamma^*}(x)$.
Proposition \ref{prop:sqrt-per-emb} implies that $\sqrt{\zeta_k} \in W^{s,2}(\T^d)$ for all $k$. Moreover, it implies also that the $\Gamma^*$ periodic functions $\sqrt{\gamma_k}$ are locally in $W^{s,2} (\R^d)$  for each $k \in \{1,...,K\}$. Indeed, if we let $A$ be the full rank matrix so that $\Gamma^* = A\Z^d$ then we have $P_{\Gamma^*}(A x)= P(H\circ A)(x)$, where $H\circ A$ is defined by $H\circ A(x)=H(Ax)$.  If $H \subset W^{s,2}(\R^d)$, then $H\circ A \subset W^{s,2}(\R^d)$ and the claim follows from applying Proposition \ref{prop:sqrt-per-emb} to $P(H\circ A)(x)$.

Next, equation \eqref{eqn:RankFormula} may be reformulated as
\[ \max \{ j: \zeta_j(x)>0\} = \sum_{k \in R}  \max\{\ell: \gamma_\ell(x+k)>0\}, \ \ \ \text{a.e. } x \in \R^d.\]
By equation \eqref{eqn:MinGen}, $J$ is the largest index such that $\zeta_j$ is not equal to zero almost everywhere.  We can assume without loss of generality that $\zeta_{J}>0$ almost everywhere, or else the proposition follows trivially.  Thus, we have,
\[ J= \sum_{k \in R} \max\{\ell: \gamma_\ell(x+k)>0\}, \ \ \ \text{a.e. } x \in \R^d.\]

Let $M$ be the largest index such that $\gamma_M$ is not identically equal to zero.  Note that $\gamma_M$ cannot be positive almost everywhere, or else we would have $J= M [\Gamma: \Z^d]$, and $[\Gamma: \Z^d]$ would divide $J$.

We first consider the case $d=1$. The following argument was given for a single generator in \cite{ASW} and extended to several generators in \cite{TW}, we repeat it for completeness.   Since $s>1/2$, Theorem \ref{thm:SobEmb} implies that $\gamma_M$ and $\zeta_J$  both have continuous representatives, and can therefore be assumed to be continuous. Let $S_{\Gamma}=\{ x \in \R: \gamma_M(x)>0\}$ and note that $S_{\Gamma}$ is open.  By the observations above $0<|S_{\Gamma}|, |S_{\Gamma}^c|$ and so there exists $x_0\in\R$ in the boundary of $S_\Gamma$ such that $\gamma_M(x_0)=0$. Lemma \ref{ev-estimate} implies that $0 \le \zeta_J(x) \le \gamma_M(x)$ on $S_\Gamma$ and due to the continuity of both functions, it follows that $\zeta_J(x_0)=0$ as well.
By the periodicity of $\zeta_J$, we can assume that this zero is in $[-1/2,1/2)$ which implies that $\mathcal{H}^{0}(\Sigma(\zeta_J))>0$.

Next, we consider the case where $d>1$. We assume that for $k=1,...,K$ we have $\gamma_M(x)=\gamma_M^*(x)$  and $\zeta_J(x)=\zeta_J^*(x)$ for every $x\in\R^d$, where $\gamma_M^*, \zeta_J^*$ are the representatives of $\gamma_M$ and $\zeta_J$ given in (\ref{fstar-defs}). In particular this implies that $\zeta_J$ and $\gamma_M$ are well defined for every $x\in\R^d$ and that they satisfy Proposition \ref{prop:HolderEmbedLines}.

Let $S_{\Gamma}=\{ x \in \R^d: \gamma_M(x)>0\}$. Note that $S_{\Gamma}$ are well defined. We have $0<|S_{\Gamma}|, |S_{\Gamma}^c|$ and so Lemma \ref{linelemma} may be applied. Let $A$ and $i_0$ be as in the Lemma.

By Proposition \ref{prop:HolderEmbedLines}, $\zeta_J$ and $\gamma_M$  are continuous on almost every line parallel to the $i_0$ coordinate axis. We can therefore assume that these functions are continuous on every line $L=L_{i_0}(x)$ with $x\in A$ and, by Lemma \ref{ev-estimate}, that $0 \le \zeta_J(x) \le \gamma_M(x)$ on $L$. Repeating the same considerations as in the case $d=1$, we conclude that on every such line there exists a point $x_0$ such that $\zeta_J(x_0)=0$. Since $\zeta_J$ is periodic, we can further assume that $A\subset [-1/2,1/2)^{d-1}$ and the points $x_0$ belong to  $[-1/2,1/2)^{d}$.

As the zero set of $\zeta_J=\zeta_J^*$ in $[-1/2,1/2)^d$ is exactly $\Sigma(\zeta_J)$, we conclude that the projection of $\Sigma(\zeta_J)$ onto the hyperplane perpendicular to the $i_0$ coordinate axis contains $A$, and therefore that it has positive $(d-1)$--Lebesgue measure.  Since projections can only decrease distances, the Hausdorff dimension of $\Sigma(\zeta_J)$ must be greater than or equal to $d-1$.
\end{proof}

\subsection{Proofs for theorems \ref{thm:CqSISHD-1d} and \ref{thm:CqSISHD}}

In this subsection we give a proof of  theorems \ref{thm:CqSISHD-1d} and \ref{thm:CqSISHD}. The $q=2$ case of these theorems is given by Theorem 1.3 of \cite{HNP}, and so we will only consider $2<q\le \infty$.

\begin{proof}

First, we prove the necessary condition for Theorem \ref{thm:CqSISHD} as this implies the same for Theorem \ref{thm:CqSISHD-1d}.  Then, we prove sharpness of the two results separately.

Let $d\geq 1$ and $F=\{f_k\}_{k=1}^{K}$ be as in the corresponding Theorem. Suppose for a contradiction that  for $t=\min( 2d/q'-d+1,2)$ we have
\begin{equation}\label{intconv}\int_\R |x|^{t} |f_k(x)|^2 dx<\infty\qquad \forall k\in \{1,...,K\},\end{equation}
that is, suppose that $\widehat{F} \in [W^{s,2}(\R^d)]^K$ for $s=\min(d/q'-d/2+1/2,1)$. By Proposition \ref{prop:sqrt-per-emb} it follows that the eigenvalues $\zeta_i$ of $P:=P(\widehat{F})$ satisfy $\sqrt{\zeta_i} \in W^{s,2}(\T^d)$ for all $i \in \{1,...,K\}$.
On the other hand, Proposition \ref{prop:LargeZeroSet} implies that $\mathcal{H}^{d-1}(\Sigma(det(P^{1/2})))>0$. Finally, since $\mathcal{T}(F)$ is a minimal $(C_q)$--system, Proposition \ref{prop:CqSISFM} implies that $P^{-1/2} \in \mathfrak{M}_2^q(K)$. So, by corollary \ref{thm:ScalarValuedLargeZeroResult-matrix} we must have $s>\min(d/q'-d/2+1/2,1)$, which is a contradiction.

We now turn to prove the sharpness of Theorem \ref{thm:CqSISHD-1d}.  For $2<q\le \infty$, let $\beta < 1-\frac{2}{q}$, and let $f_{\beta}$ be the function defined in Proposition \ref{prop:CqCounterD1}. Then, in particular, $f_{\beta}$ satisfies (\ref{intconv}) for all $t<1+\beta$. Note that taking $\beta$ arbitrarily close to $1-\frac{2}{q}$, allows $1+\beta$ to become arbitrarily close to the critical exponent $2/q'$.  Consider the family $F$ containing the single function $F=\{f_{\beta}\}$ and note that $P(\widehat{f_{\beta}})$ is a scalar valued function which is equal to $(\widehat{f_{\beta}})^2$ on $[-\frac{1}{2},\frac{1}{2}]$. By applying the Fourier transform it is readily checked that $V(f_\beta)$ is translation invariant, that is, that it is invariant under any real shift.

We claim that $P(\widehat{f_{\beta}})^{-1/2} \in \mathcal{M}_2^q$, so that $\mathcal{T}(f_\beta)$ is an exact $(C_q)$--system for $V(f_\beta)$. Indeed, by the discussion following Theorem \ref{thm:ScalarValuedLargeZeroResult} it is enough to check that $P(\widehat{f_{\beta}})^{-1/2}\in L^{\frac{2q}{q-2}}(\T)$. We have
\[ \int_{\T}P(\widehat{f_{\beta}})(\xi)^{-\frac{q}{q-2}} dx =2\int_{0}^{1/2} (\frac{1}{2}-\xi)^{-\frac{\beta q}{q-2}} dx,\]
and this integral is finite since  $\beta < 1-\frac{2}{q}$.

Finally, we use the construction above to demonstrate the sharpness of the case $q=\infty$ in Theorem \ref{thm:CqSISHD}.
For any $d\in \N$ and $\beta<1$, let $F_\beta(x)=f_\beta(x_1) f_\beta(x_2)\cdots f_\beta(x_d)$.  Then, by Proposition \ref{prop:CqCounterD1}, $F_{\beta}$ satisfies (\ref{intconv}) for all $t<2$. Consider the family $F$ containing the single function $F=\{F_{\beta}\}$ and note that as above $P(\widehat{F_{\beta}})$ is a scalar valued function which is equal to $(\widehat{F_{\beta}})^2$ on $[-\frac{1}{2},\frac{1}{2}]^d$. As above, $V(F_\beta)$ is translation invariant, that is, that it is invariant under any real shift.
Since $P(\widehat{F_{\beta}})^{-1/2} \in L^2(\T^d)=\mathcal{M}_2^{\infty}$, it follows that $\mathcal{T}(F_{\beta})$ is an exact $(C_{\infty})$--system for $V(F_{\beta})$. The proof is complete.

\end{proof}

\section{Further directions of study and open problems}\label{open problems}

In this section we outline some questions related to this work.

\begin{itemize}

\item[1.]
As mentioned in the introduction, it is likely that the results in theorems \ref{thm:ScalarValuedLargeZeroResult} and \ref{thm:CqSISHD} are not sharp (see also Theorem \ref{thm:ScalarValuedLargeZeroGeneralResult} and its corollaries). Following the approach developed above, any refinement of Theorem \ref{thm:ScalarValuedLargeZeroResult} will imply in turn an improvement of the result in Theorem \ref{thm:CqSISHD}. In particular, as discussed in the introduction, it is natural to ask whether the upper bound in part i of Theorem \ref{thm:ScalarValuedLargeZeroResult} can be replaced by the bound $q\le ({d-\sigma})/({d-\sigma-s})$.
If such a result holds then Theorem \ref{thm:CqSISHD} will remain true with the bound $t\ge 2/q'$.  With a similar argument to that of the sharpness of Theorem \ref{thm:CqSISHD-1d}, one can show that no bigger bound is possible in Theorem \ref{thm:CqSISHD}.

\item[2.]
The parameters of the Sobolev space $W^{s,r}(\T^d)$ considered in this paper are restricted to satisfy $\alpha=s-d/r <1$, which implies that this space is embedded into the space of $\alpha$-H\"older continuous functions. To obtain sharp bounds in the case of $\alpha=1$, some refinement of the results is expected. In particular, such a refinement in Theorem \ref{thm:nonsymSobolevEmb} will allow to extend Theorem \ref{thm:nonsymmetricBLT} to the full boundary of the region described in \cite{NO1}.

Next, it is interesting to compare the case $\alpha >1$ to the case $\alpha=1$. Since the zeros considered in this paper are in general assumed to be of order one,
it is not clear if the added regularity will make any change to the results.

\item[3.]
The range of $(p,q)$ multipliers considered in Section \ref{subsec:PQMult} guarantees that $\mathcal{M}_p^q \subset L^2(\T^d)$.  However, when $1 \le p \le 2 \le q \le \infty$, the space $\mathcal{M}_p^q$ contains distributions which cannot be represented by functions. However, it seems natural to ask whether there exists a version of our multiplier results also for $L^1(\T^d)$ functions which belong to such spaces.

\item[4.] We list shortly some additional settings in which our approach may be applied:
\begin{itemize}
\item[i.] Does a sharp version of the Balian-Low results in \cite{NO1} exist in higher dimensions?
\item[ii.] Can the condition of 'non trivial extra invariance' in Theorem \ref{thm:CqSISHD} be replaced by the condition that the system of translates is not minimal in the space? We note that it is not hard to show that this is indeed the case in  Theorem \ref{thm:CqSISHD-1d}. (See also Theorem 1.5 in \cite{HNP}).
\item[iii.] Does there exist a version of our results in the case where the Gabor system $G(\Z)$ is not complete in $L^2(\R)$, but is a $C_q$ system in its closed linear span? (See e.g. \cite{CLPP} for the case $q=2$).
\end{itemize}

\end{itemize}

\section*{Acknowledgements}
The new results in this paper were developed in collaboration with the author's doctoral advisor, Alex Powell, and the author's postdoctoral mentor, Shahaf Nitzan.  The contents should be considered a joint work.  The author greatly thanks both of them for their guidance, time, and contributions during the course of this project.  The author would also like to thank Chris Heil and Jan--Fredrik Olsen for helpful comments and discussions.  

The author was supported by NSF grant DMS-1344199.  Work on this paper was also supported by  NSF grants DMS-1600726 and DMS-1521749.

\end{document}